\documentclass[11pt,reqno]{amsart}

\usepackage{amsmath, amsthm, amssymb}
\usepackage{amsfonts}

\usepackage{pdfsync}
\usepackage{esint}
\usepackage[ansinew]{inputenc}
\usepackage[dvips]{epsfig}
\usepackage{graphicx}
\usepackage[english]{babel}
\usepackage{hyperref}
\usepackage{color}
\usepackage[shortlabels]{enumitem}
\theoremstyle{plain}
\newtheorem{thm}{Theorem}[section]
\newtheorem{cor}[thm]{Corollary}
\newtheorem{lem}[thm]{Lemma}
\newtheorem{prop}[thm]{Proposition}

\theoremstyle{definition}
\newtheorem{defi}[thm]{Definition}

\theoremstyle{remark}
\newtheorem{rem}[thm]{Remark}

\numberwithin{equation}{section}

\newcommand{\average}{{\mathchoice {\kern1ex\vcenter{\hrule height.4pt
width 6pt depth0pt} \kern-9.7pt} {\kern1ex\vcenter{\hrule
height.4pt width 4.3pt depth0pt} \kern-7pt} {} {} }}

\newcommand{\N}{\mathbb{N}}

\newcommand{\R}{\mathbb{R}}
\newcommand{\Ss}{\mathbb{S}}

\newcommand{\vep}{\varepsilon}

\def\intrr{\mathrm{int}}

\def\dist{\mathrm{dist}}
\def\Reg{\mathrm{Reg}}

\usepackage[colorinlistoftodos]{todonotes} 
\begin{document}

\title[Regularity theory for fully nonlinear parabolic obstacle problems]{Regularity theory for fully nonlinear parabolic obstacle problems}

\author{Alessandro Audrito}
\address{Alessandro Audrito \newline \indent 
Department of Mathematics, ETH Zurich, R\"amistrasse 101, 8092 Zurich, Switzerland}
\email{alessandro.audrito@math.ethz.ch}

\author{Teo Kukuljan}
\address{Teo Kukuljan \newline \indent
Universitat de Barcelona, Departament de Matem\`atiques i Inform\`atica, Gran Via de les \newline \indent 
Corts Catalanes 585, 08007 Barcelona, Spain.}
\email{tkukuljan@ub.edu}

\thanks{TK is supported by the European Research Council (ERC) under the Grant Agreement No 801867 and the AEI project PID2021-125021NA-I00 (Spain).
AA is supported by the European Union's Horizon 2020 research and innovation programme under the Marie Sk{\l}odowska--Curie grant agreement 892017 (LNLFB-Problems). \\
We wish to thank Prof. X. Ros-Oton for his constant support and inspiring discussions concerning the results presented in the paper.
}

\keywords{Free boundary, fully nonlinear diffusion, Obstacle problem, regularity.}

\subjclass[2010]{35R35, 35K55, 35B44, 35B65.}

\begin{abstract}
We study the free boundary of solutions to the parabolic obstacle problem with fully nonlinear diffusion. We show that the free boundary splits into a regular and a singular part: near regular points the free boundary is $C^\infty$ in space and time. Furthermore, we prove that the set of singular points is locally covered by a Lipschitz manifold of dimension $n-1$ which is also $\varepsilon$-flat  in  space, for any $\varepsilon>0$.
\end{abstract}

\maketitle


%
%
%
%
%
\section{Introduction}
The classical parabolic obstacle problem can be formulated as follows:
\begin{equation*}
\begin{cases}
\partial_tu - \Delta u = -\chi_{\{u > 0\}} &\text{in } Q_1\subset \R^{n+1} \\
u \geq 0, \; \partial_t u \geq 0             &\text{in } Q_1.
\end{cases}
\end{equation*}
It arises in many fields and applications such as the study of phase transitions in the Stefan problem and optimal stopping (see e.g. \cite{F18} and \cite{E13}, respectively) and it is one of the most remarkable examples in the class of \textit{free boundary} problems. Indeed, the positivity set of the solution and its topological boundary are \emph{not} fixed a priori, but are \textit{unknowns} of the problem. For this reason, the main goal is to describe as precisely as possible both the solution $u$ and its free boundary $\partial\{u>0\}$.

For what concerns the solution, one can establish the optimal regularity ($C^{1,1}$ is space and $C^1$ in time) plus additional qualitative/quantitative properties like non-degeneracy and semi-convexity in space (see \cite{Caf78,Caf78Bis,CF79}). However, a much more
 challenging and delicate issue is to investigate the regularity of the free boundary, namely to answer the following question:
$$\text{Is the free boundary } C^\infty?$$
The regularity theory for the free boundary was developed by Caffarelli in his 
groundbreaking paper \cite{Caf78}, where he showed that the free boundary can be split into \textit{regular points} and \textit{singular points}: the regular points form an open subset of the free boundary which is locally $C^\infty$, while the singular point enjoy some \emph{stratification} properties. The set of singular points was later studied in \cite{B1anchet06,B1anchet06Bis,BlaDolMon06}, where the authors proved the uniqueness of blow-ups at singular points and the $C_t^1$ regularity of solutions for general linear parabolic operators with smooth coefficients (see also \cite{Caf78Bis,CF79} for the standard diffusion framework), and \cite{LM15}, where the authors showed that the singular set can be locally covered with a $C^1_x\cap C^{1/2}_t$ manifold of dimension $n-1$ (see also \cite{B01,B02} for a covering of class $C^{0,1}$). Only recently, Figalli, Ros-Oton and Serra (\cite{FigRosSerra1}) gave a sharp bound on the parabolic Hausdorff dimension of the singular set and a $C^\infty$-expansion of solutions near singular points, up to a set of higher codimension. In particular, when $n = 3$, it turns out that the free boundary is smooth (in space) for a.e. time. It is important to mention that the assumption $\partial_tu \geq 0$ plays a crucial role in the whole analysis: many of the results mentioned above are not known when the time-monotonicity of solutions is not assumed a priori. We quote \cite{CafPetSha04} by Caffarelli, Petrosyan and Shahagholian for the analysis of the free boundary without any assumption on the sign of both $u$ and $\partial_tu$.

The goal of this paper is to study the parabolic obstacle problem with fully nonlinear diffusion:
\begin{equation}\label{eq:ProblemIntro}
\begin{cases}
\partial_tu - F(D^2u,x) = f(x) \chi_{\{u > 0\}} &\text{in } Q_1 \\
u \geq 0, \; \partial_t u \geq 0             &\text{in } Q_1,
\end{cases}
\end{equation}
where $F: \mathcal{S} \times B_1 \to \R$ and $f: B_1 \to \R$ are given functions and $\mathcal{S}$ denotes the linear space of $n \times n$ symmetric matrices. We assume that $F$ satisfies the following conditions
\begin{equation}\label{eq:AssOnFIntro}
\begin{cases}
F \in C^\infty \\
F(\cdot,x) \text{ is convex for all } x \in B_1, \\
F \text{ is uniformly elliptic}, \\
F(O,x) = 0 \text{ for all } x \in B_1.
\end{cases}
\end{equation}
We say that $F$ is \emph{uniformly elliptic} if there exist $0 < \lambda \leq \Lambda$ such that
\[
\lambda \| N \| \leq F(M+N,x) - F(M,x) \leq \Lambda\| N \|,
\]
for all $x \in B_1$, all $M \in \mathcal{S}$, and all $N \in \mathcal{S}$ satisfying $N \geq 0$; see \cite{CC95,FR21,IS12,Lieberman1996} for a comprehensive treatise about fully nonlinear uniformly elliptic operators. We  also assume that  
\begin{equation}\label{eq:AssOnfIntro}
f \in C^\infty \text{ and } f \leq - c_\circ,
\end{equation}
for some $c_\circ > 0$.\footnote{The assumptions $F,f \in C^\infty$ are not needed in the majority of our statements but, at this stage, useful to simplify the presentation. Weaker assumptions on the regularity of $F$ and $f$ will be given later on in the paper.}

In the stationary version of the problem
\[
F(D^2u,x) =  \chi_{\{u > 0\}}
\]
the optimal regularity of solutions and the $C^{1,\alpha}$ regularity of the regular part of the free boundary was studied by Lee in his PhD thesis \cite{Lee1998}. The optimal regularity was investigated in a more general setting by Figalli and Shahgholian in \cite{FigShah}, where they additionally proved that if the free boundary is Lipschitz, then it is $C^1$. They extended such results to the parabolic setting in \cite{FigShahBis}, see also \cite{IM16}. We also quote \cite{IN19} (elliptic framework), where the author shows that the free boundary can be locally written as the graph of a $C^1$ function up to the fixed boundary: in particular, this shows that the free boundary leaves the fixed boundary in a tangential way (see also \cite{IN19Bis}).

The higher regularity of the free boundary in both elliptic and parabolic setting is a consequence of the celebrated work by Kinderlehrer and Nirenberg \cite{KN77}. However, notice that there was still a small gap between the $C^1$ regularity and the higher regularity, since in \cite{KN77} the solution is assumed to be $C^2$ in the positivity set up to the boundary, while in the above papers the solutions are only proved to be $C^{1,1}$.

The singular set has been studied in \cite{B01,B02} (fully nonlinear elliptic setting), where Bonorino showed that the singular set can be locally covered with a Lipschitz $(n-1)$-dimensional manifold, and in \cite{SY21} where Savin and Yu improved the regularity of the covering manifold to $C^{1,\log^\varepsilon}$, see also \cite{SY22}. For the parabolic case (with fully nonlinear diffusion) no results were known for singular points. We finally mention the papers \cite{PetShah,Shah} for further analysis of the free boundary in fully nonlinear obstacle problems with non-smooth obstacles and some applications to finance.

\subsection{General strategy and main results} As mentioned above, in this paper we study the free boundary of solutions to \eqref{eq:ProblemIntro}, when the standard diffusion is replaced by the fully nonlinear one, in the sense of \eqref{eq:AssOnFIntro}.

The general strategy we follow is well-known by the experts in free boundary problems and is based on the asymptotic analysis of the so-called \emph{blow-up families}. A blow-up family of $u$ at $(x_0,t_0) \in \partial\{u > 0\}$ is a normalised re-scaling of $u$ centred at $(x_0,t_0)$, defined by the formula
\begin{equation}\label{eq:ParNormRescaling}
u^{(x_0,t_0)}_r(x,t) := \frac{u(x_0 + rx,t_0 + r^2t)}{r^2}, \qquad r \in (0,1).
\end{equation}
We will write $u_r$, when $(x_0,t_0) = (0,0)$. Notice that the process of blowing-up ``zooms-in'' the solution around free boundary points and, moreover, thanks to the normalisation factor $r^{-2}$, it ``preserves'' the equation: indeed, notice that each re-scaling in \eqref{eq:ParNormRescaling} satisfies
\begin{equation}\label{eq:RescalingEqn}
\partial_t v - F_r^{x_0}(D^2v,x) = f_r^{x_0}(x)\chi_{\{v > 0\}} \quad \text{ in } Q_{1/r},
\end{equation}
where $F_r^{x_0}(M,x) := F(M,x_0 + rx)$ and $f_r^{x_0}(x) := f(x_0 + rx)$. In light of these heuristic observations, one wants to study the limit as $r \downarrow 0$ and prove that the local behaviour of solutions to \eqref{eq:ProblemIntro} can be described in terms of the local behaviour of \emph{blow-up limits} (or, shortly, \emph{blow-ups}), i.e., limits of blow-up families along some sequence $r_k \downarrow 0$:
\[
u_{r_k}^{(x_0,t_0)} \to u_0^{(x_0,t_0)},
\]
as $k \to \infty$, in some suitable topology. The rough idea is that some of the properties of the blow-ups are shared by the blow-up families (for small $r$'s) and so, in turn, by the solution itself at very small scales. For instance, if $(0,0) \in \partial\{u > 0\}$, can we prove that
\[
\partial \{ u_r > 0 \} \text{ is locally } C^\infty \text{ for } r \sim 0, \text{ if } \partial \{ u_0 > 0 \} \text{ is locally } C^\infty ?
\]
This plan presents several difficulties and is developed in steps as follows. First, to prove the mere existence of blow-ups (along some suitable sequence), one has to establish \emph{compactness} properties for blow-up families such as optimal growth and optimal regularity of solutions (see Lemma \ref{lem:OptimalGrowth} and Theorem \ref{thm:OptimalReg}). Further, additional properties of solutions like non-degeneracy, semi-convexity in space and $C_t^1$ regularity are required in order to derive the blow-ups equation (see Lemma \ref{lem:NonDegeneracy} and Proposition \ref{prop:semicnovexity}): as a consequence of these properties, blow-ups turn out to be time independent, convex and satisfy an elliptic problem in the whole space (see \eqref{eq:BlowUpEqn}). It is then crucial to \emph{classify} blow-ups according to the notion of \emph{regular} and \emph{singular} points given in Definition \ref{def:RegSingPoints}. In this part, we basically adapt the ideas of Caffarelli \cite{Caf78}.

At this point, the analysis of the regular part of the free boundary, denoted by $\Reg(u)$, begins. It is divided in three main steps we summarize as follows:

\smallskip

$\bullet$ $\Reg(u)$ is locally $C^{0,1}$;

$\bullet$ If $\Reg(u)$ is locally $C^{0,1}$, then $\Reg(u)$ is locally $C^1_x \cap C^{0,1}_t$;

$\bullet$ If $\Reg(u)$ is locally $C^1_x \cap C^{0,1}_t$, then $\Reg(u)$ is locally $C^\infty$.

\smallskip

\noindent  We anticipate that, in the second step, it will be essential to show that $u$ is locally $C^2_x$ up to the free boundary near regular points (see Proposition \ref{prop:C1ofFB}). This property will allow us to bootstrap the regularity by means of some higher order boundary Harnack estimates from \cite{K21} and deduce the (local) smoothness of $\Reg(u)$ (see the proof of Theorem \ref{thm:SmoothnessFB}). Alternatively, we may apply \cite[Theorem 3]{KN77} by Kinderlehrer and Nirenberg.

Finally, we turn our attention to the singular part, denoted with $\Sigma(u)$. In the spirit of \cite{B01} and assuming that the function $F$ is independent of $x$, we show that singular points can be locally covered by a Lipschitz manifold of  dimension $n-1$ and, as a consequence of \cite[Corollary 7.8]{FigRosSerra}, we also obtain that almost every time-slice of $\Sigma(u)$ has at most Hausdorff dimension $n-2$. The main ingredients in this part of the analysis are two ``improved'' bounds: one for the time derivative (see Lemma \ref{lem:towardsLipschitzness}) and one for the second spatial derivatives along the main direction of growth of the blow-ups (see Lemma \ref{lem:positivityOfSecondDer}). The first estimate allows us to prove that the free boundary can be written as a graph, where the time variable is a Lipschitz function depending on the space ones. The second yields the $\varepsilon$-flatness in space of the Lipschitz covering, for any $\varepsilon>0$.

This is the strategy we follow through the rest of the paper. Now, we state our main results, starting with the classification of blow-ups and the $C^\infty$ regularity of $\Reg(u)$.
\begin{thm}\label{thm:1.1}
	Let $u$ be a solution to \eqref{eq:ProblemIntro} with $F,f$ satisfying \eqref{eq:AssOnFIntro}-\eqref{eq:AssOnfIntro}. Then, for every free boundary point $(x_0,t_0) \in \partial\{ u > 0 \}$ it holds
	\begin{enumerate}[(i)]
		\item either there are $c > 0$ and $e \in \mathbb{S}^{n-1}$ such that
		\[
		u_r^{(x_0,t_0)} \to c(e\cdot x)_+^2 \quad \text{ in } C^{1,\alpha}_x \cap C^\alpha_t \text{ locally in } \R^{n+1},
		\]
		as $r \to 0$,
		
		\item or for every sequence $r_k \to 0$, there is a subsequence $r_{k_j}$ and a matrix $A \geq 0$ such that
		\[
		u_{r_{k_j}}^{(x_0,t_0)} \to x^T A x \quad \text{ in } C^{1,\alpha}_x \cap C^\alpha_t \text{ locally in } \R^{n+1},
		\]
		as $j \to +\infty$.
	\end{enumerate}
	The points where $(i)$ holds are called regular. The set of regular points form an open subset of $\partial\{u>0\}$ which is locally $C^\infty$. Points where $(ii)$ holds are called singular.
\end{thm}
As already mentioned, we also investigate the properties of the singular set, as stated in the following theorem.
\begin{thm}\label{thm:1.2}
Let $u$ be a solution to \eqref{eq:ProblemIntro} with $F$ independent of $x$ and satisfying \eqref{eq:AssOnFIntro}-\eqref{eq:AssOnfIntro}. Assume that $\partial_tu>0$ in $\{u>0\}$ and let $\Sigma(u) \subset \R^n\times\R$ be the set of singular points. Then for any $\varepsilon>0$, $\Sigma(u)$ can be locally covered by a Lipschitz manifold of dimension $n-1$, which is $\varepsilon$-flat in space.\footnote{That means that the manifold can be expressed as a graph of a Lipschitz function whose Lipschitz semi-norm in space is smaller than $\varepsilon$.} 
\end{thm}
Theorem \ref{thm:1.2} is sharp in the sense of the dimension of the covering manifolds, as examples can be constructed where the singular set is indeed of dimension $n-1$, see \cite{FigRosSerra1}.

Combining the above statement with \cite[Corollary 7.8]{FigRosSerra} we deduce that for almost every time the singular set must have at most Hausdorff dimension $n-2$.
\begin{cor}\label{cor:1.3}
	Let $u$ be a solution to \eqref{eq:ProblemIntro} with $F,f$ satisfying \eqref{eq:AssOnFIntro}-\eqref{eq:AssOnfIntro}. Assume that $\partial_tu > 0$ in $\{u > 0\}$ and let $\Sigma(u)_t$ the set of singular points at time $t$. Then:
	\[ 
	\operatorname{dim}_{\mathcal{H}} \left(\Sigma(u)_t\right)\leq n-2,
	\]
	where $\operatorname{dim}_{\mathcal{H}}(E)$ denotes the Hausdorff dimension of a set $E$.
\end{cor}
As a consequence of our main results, the smoothness of the regular part and the structure of the whole free boundary are completely understood. However, in the study of the singular set, many problems are still open. For instance: are the blow-ups unique at singular points? Can we improve the regularity of the covering manifold? Can we prove sharp bounds on the \emph{parabolic} Hausdorff dimension of the singular set? Is the free boundary \emph{generically} smooth, in the sense of \cite{FigRosSerra1}?

These delicate and challenging problems will be at the centre of our future research.

\subsection{Structure of the paper}
The paper is organized as follows. In Section \ref{sec:notation} we present the notations we use throughout the paper, as well as the definitions of the parabolic H\"older spaces. In Section \ref{sec:preliminaries} we state some preliminary results about fully nonlinear equations. In Section \ref{sec:Existence} we prove existence of solutions. In Section \ref{sec:optimalRegularity} we prove the optimal growth and the almost optimal $C^{1,1}_x\cap C^{0,1}_t$ regularity of solutions, while in Section \ref{sec:semiConvexity} we establish semi-convexity estimates and the continuity of the time derivative. Blow-ups are classified in Section \ref{sec:blowUps}, while Section \ref{sec:regularPoints} is devoted to the study of the free boundary near regular points. Finally, in Section \ref{sec:singularPoints} we study the singular part of the free boundary. At the very end there is an appendix, where we prove some of the results of Section \ref{sec:preliminaries}.

%
%
%
%
%
\section{Notations and definitions}\label{sec:notation}
Below we present the notations and definitions we use throughout the paper.
\subsection{Notations}
\begingroup
\allowdisplaybreaks
\begin{align*}
& B_r(x_0) = \{x \in \mathbb{R}^N: |x-x_0| < r \} \\
& \tilde{B}_r(x_0) = \{x \in \mathbb{R}^N: |x-x_0|_\infty < r \}, \qquad  |x|_\infty := \max \{|x_i|: i =1,\ldots,n\} \\
& Q_r(x_0,t_0) = B_r(x_0) \times (t_0-r^2,t_0+r^2)\\
& Q_r^+(x_0,t_0) := B_r(x_0) \times (t_0,t_0+r^2) \\
& Q_r^-(x_0,t_0) := B_r(x_0) \times (t_0-r^2,t_0) \\
& D_r^+(x_0,t_0) := B_r(x_0) \times (t_0 + 3r^2,t_0 + 4r^2) \\
& D_r^-(x_0,t_0) := B_r(x_0) \times (t_0 - 3r^2,t_0 - 2r^2) \\
& \tilde{Q}_r(x_0,t_0) := \tilde{B}_r(x_0) \times (t_0-r^2,t_0+r^2) \\
& \tilde{D}_r^+(x_0,t_0) := \tilde{B}_r(x_0) \times (t_0 + 3r^2,t_0 + 4r^2) \\
& \tilde{D}_r^-(x_0,t_0) := \tilde{B}_r(x_0) \times (t_0 - 3r^2,t_0 - 2r^2) \\
& d((x,t),(y,\tau)) := \sqrt{|x-y|^2 + |t-\tau|} \quad \text{ or } \quad d((x,t),(y,\tau)) := \max\{|x-y|,|t-\tau|^{1/2}\} \\
& d((x,t),A) := \inf_{(y,\tau) \in A} d((x,t),(y,\tau)), \quad A \subset \R^{n+1} \\
&N_\delta(A) = \{(x,t) \in \R^{n+1}: \dist((x,t),A) < \delta \}), \quad A\subset \R^{n+1} \\
&\intrr (A) := \overline{A} \setminus \partial A, \quad A \subset \R^{n+1} \\
& a_{ij}\partial_{ij}v := \sum_{i,j=1}^n a_{ij}\partial_{ij}v, \quad \{a_{ij}\}_{i,j=1}^n \;\; n \times n \text{-matrix} 
\end{align*}
\endgroup
\subsection{Definitions} We first introduce the following class of solutions, in the spirit of \cite{CafPetSha04}.
\begin{defi}\label{def:SolFNL} (Class of solutions) 
	\noindent Let $r,K > 0$ and $(x_0,t_0) \in \R^{n+1}$. We say that $u \in \mathcal{P}_r(x_0,t_0;K)$ if 
	
	\smallskip
	
	$\bullet$ $|u| + |f| \leq K$ in $Q_r(x_0,t_0)$;
	
	$\bullet$ $u \in C^{1,\alpha}_p $ and $u \in W^{2,p}_x \cap W_t^{1,p}$ locally in $Q_r(x_0,t_0)$, for all $\alpha \in (0,1)$ and $p \in (1,\infty)$;

	$\bullet$ $u$ satisfies \eqref{eq:ProblemIntro} a.e. in $Q_r(x_0,t_0)$.
	
	\smallskip
	
	\noindent When $(x_0,t_0)=(0,0)$ we write $\mathcal{P}_r(K)$. 
\end{defi}
Next, we define a class of parabolic H\"older spaces, following the definition given in \cite{K21}. 

\begin{defi}
	For a multi-index $\alpha\in \N_0^{n+1}$, we denote $\alpha_x$ the first $n$ components and $\alpha_t$ the last one, and define
	$$|\alpha|_p = |\alpha_x| + 2\alpha_t = \sum_{i=1}^n\alpha_i + 2\alpha_t.$$
	Furthermore, we define the parabolic derivatives with $D^k_p = \{\partial^\alpha:\text{ }|\alpha|_p = k\}$ 	and the parabolic polynomial spaces as follows:
	$$\textbf{P}_{k,p} = \left\{ \sum_{|\alpha|_p\leq k}c_\alpha (x,t)^\alpha;\text{ }c_\alpha \in\R\right\}.$$
	We say that $k$ is the parabolic degree of polynomial $q$, if $k$ is the least integer so that $q\in\textbf{P}_{k,p}.$
	For a polynomial $q = \sum_\alpha c_\alpha  (x,t)^\alpha$, we denote
	$$||q|| = \sum_\alpha|c_\alpha|.$$ 
\end{defi}

\begin{defi}\label{def:parabolicHolderSpaces} 

	Let $\Omega$ be an open subset of $\R^{n+1}$. For $\alpha\in(0,1]$ we define the parabolic H\"older seminorm of order $\alpha$ as follows
	$$\left[ u\right]_{C^\alpha_p(\Omega)} := \sup_{(x,t),(y,s)\in\Omega}\frac{|u(x,t)-u(y,s)|}{|x-y|^\alpha+|t-s|^{\frac{\alpha}{2}}},$$
	and
	$$\left[ u\right]_{C^{\alpha}_t(\Omega)} := \sup_{(x,t),(x,s)\in\Omega}\frac{|u(x,t)-u(x,s)|}{|t-s|^{\alpha}}.$$
	Furthermore, we set
	\begin{align*}
	\left[ u\right]_{C^{1,\alpha}_p(\Omega)}& :=  \left[ \nabla_x u\right]_{C^{\alpha}_p(\Omega)}+\left[ u\right]_{C^{\frac{1+\alpha}{2}}_t(\Omega)}.
	\end{align*}
	For indexes $k\geq 2$, we set
	\begin{align*}
		\left[ u\right]_{C^{k,\alpha}_p(\Omega)} := &\left[ \nabla_x u\right]_{C^{k-1,\alpha}_p(\Omega)}+\left[\partial_t u\right]_{C^{k-2,\alpha}_p(\Omega)}\\
		=& \sum_{|\gamma|_p=k} \left[ \partial^\gamma u \right]_{C^\alpha_p(\Omega)} + \sum_{|\gamma|_p=k-1} \left[ \partial^\gamma u\right]_{C^{\frac{1+\alpha}{2}}_t(\Omega)}.
	\end{align*}
	We say that $u\in C^{k,\alpha}_p(\Omega)$, when $\left[ u\right]_{C^{k,\alpha}_p(\Omega)}<\infty$, and define 
	$$||u||_{C^{k,\alpha}_p(\Omega)} := \sum_{i\leq k} ||D^i_p u||_{L^\infty(\Omega)} + 	\left[ u\right]_{C^{k,\alpha}_p(\Omega)}.$$
	
	For $k\in\N$, we say that $u\in C^k_p(\Omega)$, if there exists a modulus of continuity $\omega\colon \left[ 0, \infty\right) \to \left[ 0,\infty\right) $ -- a continuous, increasing function with $\omega(0)=0$ -- so that for all $(x,t),(y,s)\in\Omega$
	$$|D^k_p u(x,t) - D^k_p u(y,s)|\leq \omega(|x-y|+|t-s|^{\frac{1}{2}}),$$
	and
	$$|D^{k-1}_p u(x,t) - D^{k-1}_p u(x,s)|\leq |t-s|^{\frac{1}{2}}\omega(|t-s|^{\frac{1}{2}}).$$
	We set
	$$||u||_{C^k_p(\Omega)} := \sum_{l\leq k} ||D^l_p u||_{L^\infty(\Omega)} .$$
	We will often write $C^{k+\alpha}$ instead of $C^{k,\alpha}.$
	
\end{defi} 
\begin{rem}
	The parabolic H\"older space of order $\beta > 0$ introduced above is made of functions which can be approximated by polynomials $\textbf{P}_{\lfloor\beta\rfloor,p}$, up to an error of order $\beta$. More precisely, if $u\in C^\beta_p(\Omega)$ for some $\beta>0$, then for every $(x_0,t_0)\in \Omega$, there is a polynomial $p_{(x_0,t_0)}\in \textbf{P}_{\lfloor\beta\rfloor,p}$ such that
	$$|u(x,t)-p_{(x_0,t_0)} (x,t)|\leq C (|x-x_0|^{\beta}+|t-t_0|^\frac{\beta}{2}),$$
	in a neighbourhood of $(x_0,t_0)$, for some $C > 0$ depending on $u$, see \cite[Chapter IV.]{Lieberman1996} for more details.
\end{rem}
%
%
%

%
%
%
%
%
\section{Preliminaries}\label{sec:preliminaries}
In this section we present some known results about parabolic equations with fully nonlinear diffusion which we will exploit later on in the paper. We begin with the Harnack inequality for solutions to fully nonlinear parabolic equations: the symbols $\mathcal{M}^+,\mathcal{M}^-$ denote the Pucci extremal operators for ellipticity constants $\lambda$ and $\Lambda$. For the definition of Pucci extremal operators we refer to \cite{FR21}.
\begin{thm}[\protect{Harnack inequality \cite[Theorem 4.32]{IS12}}]\label{lem:Harnack} 
	Let $r > 0$, $(x_0,t_0) \in \R^{n+1}$ and let $u\geq 0$ satisfy 
	\begin{equation}\label{eq:FNL}
	\begin{cases}
	\partial_t u - \mathcal{M}^+(D^2u) \leq C_0 &\text{ in }Q_r(x_0,t_0)\\
	\partial_t u - \mathcal{M}^-(D^2u) \geq -C_0 &\text{ in }Q_r(x_0,t_0).
	\end{cases}
	\end{equation}
	Then there exists $C > 0$ depending only on $n$, $\lambda$ and $\Lambda$ such that\footnote{The statement is slightly different from \cite{IS12}, but the proof is exactly the same.}
	\begin{equation*}
	\sup_{D_{r/2}^-(x_0,t_0)} u \leq C \bigg( \inf_{D_{r/2}^+(x_0,t_0)} u + r^2 C_0 \bigg).
	\end{equation*}
\end{thm}
When considering nonnegative supersolutions instead of solutions, we can control the infimum with the $L^p$ average. 
\begin{thm}[\protect{Weak Harnack inequality \cite[Theorem 4.15]{IS12}}]\label{thm:basicWeakHarnack}  
Let $r > 0$, $(x_0,t_0) \in \R^{n+1}$ and let $u\geq0$ satisfy
$$\partial_tu - \mathcal{M}^-(D^2u)\geq -C_0 \quad \text{ in }Q_r(x_0,t_0).$$
Then there exist $C > 0$ and $p \in (0,1)$ depending only on $N$, $\lambda$ and $\Lambda$ such that
\begin{equation}\label{eq:WeakHarnack}
\bigg( \fint_{D_{r/2}^-(x_0,t_0)} u^p \bigg)^{\frac{1}{p}}\leq C\bigg(\inf_{D_{r/2}^+(x_0,t_0)} u + r^2C_0\bigg).
\end{equation}

\end{thm}
Sometimes this result is also called the half-Harnack inequality, as for subsolutions it holds that the supremum is controlled by the $L^p$ average (see \cite[Proposition 4.34]{IS12}).
We also need the following version of the weak Harnack inequality, that quantifies the growth of the constant, as the set on which we compute the $L^p$ average approaches the boundary.

\begin{lem}\label{lem:iteratedWeakHarnack}
	   
	Let $\delta \in (0,\tfrac{1}{4})$, $r > 0$, $(x_0,t_0) \in \R^{n+1}$ and let $u\geq0$ satisfy
	$$\partial_tu - \mathcal{M}^-(D^2u)\geq -C_0 \quad\text{ in }Q_r(x_0,t_0).$$
	Then there exist $C,m > 0$, $p \in (0,1)$ depending only on $n$, $\lambda$ and $\Lambda$ such that
	\begin{equation}\label{eq:ParHalfHarnackdel}
	\bigg( \fint_{Q_{\delta r}(x,t)} u^p \bigg)^{\frac{1}{p}}\leq \frac{C}{\delta^m} \bigg(\inf_{Q_{r/2}^+(x_0,t_0)} u + r^2C_0\bigg),
	\end{equation}
	for all $(x,t) \in B_{(1-2\delta)r}(x_0) \times (t_0 - (1-4\delta^2)r^2, t_0 - \tfrac{3}{4}r^2)$.
	 
\end{lem}
We postpone the proof to the appendix, due to its technical nature.
\begin{rem}\label{rem:ParHalfHarnParCubes}
	We remark that the same statement remains valid when estimating the $L^p$ norm of supersolutions in ``parabolic cubes'' $\tilde{Q}_{\delta r}(x,t)$: under the assumptions of Lemma \ref{lem:iteratedWeakHarnack}, there holds 
	\begin{equation}\label{eq:ParHalfHarnackdelParCubes}
	\bigg( \fint_{\tilde{Q}_{\delta r}(x,t)} u^p \bigg)^{\frac{1}{p}}\leq \frac{C}{\delta^m} \bigg(\inf_{Q_{r/2}^+(x_0,t_0)} u + r^2 C_0\bigg),
	\end{equation}
	for all $(x,t) \in B_{(1-2\delta)r}(x_0) \times (t_0 - (1-4\delta^2)r^2, t_0 - \tfrac{3}{4}r^2)$.

	To see this, we fix $r=1$, $(x_0,t_0) = (0,0)$, $\delta \in (0,\tfrac{1}{4})$ and $(x,t) \in B_{1-2\delta} \times (-1 + 4\delta^2, -\tfrac{3}{4})$. We define
	\[
	\varrho := \inf\{ \rho > 0: \tilde{Q}_\delta(x,t) \subset Q_\rho(x,t)\}.
	\]
	By construction, $\varrho = c_n \delta$, for some $c_n > 0$ (depending only on $n$). Consequently, by \eqref{eq:ParHalfHarnackdel}
	\[
	\bigg( \fint_{\tilde{Q}_{\delta}(x,t)} u^p \bigg)^{\frac{1}{p}} \leq \bigg( \frac{|Q_\varrho|}{|\tilde{Q}_{\delta}|}\fint_{Q_\varrho(x,t)} u^p \bigg)^{\frac{1}{p}} \leq \frac{C}{\delta^m} \bigg(\inf_{Q_{1/2}^+} u + C_0\bigg),
	\]
	and \eqref{eq:ParHalfHarnackdelParCubes} follows.
\end{rem}

We also need a version of the weak Harnack inequality, where the average is not computed over a cylinder or a cube.

\begin{cor}\label{cor:weakHarnack}
	 
	  Let $u\geq0$ satisfy
	  $$\partial_tu - \mathcal{M}^-(D^2u)\geq -C_0 \quad\quad\text{ in }Q_r(x_0,t_0).$$ 
	  Let $\delta \in (0,\tfrac{1}{4})$ and $C_0,m_0 > 0$. Then there exist $C,m > 0$, $p \in (0,1)$ depending only on $n$, $\lambda$, $\Lambda$, $C_0$ and $m_0$ such that 
	\begin{equation}\label{eq:ParHalfHarnackdelBis}
	\bigg( \fint_A u^p \bigg)^{\frac{1}{p}}\leq \frac{C}{\delta^m} \bigg(\inf_{Q_{r/2}^+(x_0,t_0)} u + C_0\bigg),
	\end{equation}
	for every open set $A \subset B_{1-2\delta} \times (t_0-(1 - 4\delta^2)r^2,t_0 -\tfrac{3}{4}r^2)$ satisfying $|A| \geq C_0\delta^{m_0}r^{n+2}$. 
\end{cor}
The proof is in the appendix. We end this preliminary section with a simple, yet important remark.
\begin{rem}\label{rem:Scaling} Let $u$ be a solution to \eqref{eq:FNL}. Then, the function 
	\begin{equation}
	\tilde{u}(x,t) = u(x_0 + rx,t_0 + r^2t)
	\end{equation}
	is a solution to
	\[
	\begin{cases}
	\partial_t u - \tilde{\mathcal{M}}^+(D^2u) \leq r^2C_0 &\text{ in } Q_1\\
	\partial_t u - \tilde{\mathcal{M}}^-(D^2u) \geq  r^2C_0 &\text{ in }Q_1,
	\end{cases}
	\]
	where
	\[
	\tilde{\mathcal{M}}^{\pm}(M) := r^2 \mathcal{M}^{\pm}(r^{-2}M).
	\]
One can easily verify that both $\tilde{\mathcal{M}}^+$ and $\tilde{\mathcal{M}}^-$ satisfy the conditions in \eqref{eq:AssOnFIntro} uniformly w.r.t. $r > 0$ and $(x_0,t_0)$.	
\end{rem}

%
%
%
%
%
\section{Existence of solutions via penalization method}\label{sec:Existence}
The classical parabolic obstacle problem can be formulated in two equivalent ways:
	\[
	\begin{cases}
	\partial_t v - \Delta v \geq 0 \\
	\partial_t v - \Delta v = 0 \quad \text{ in } \{v > \varphi\} \\
	v \geq \varphi,
	\end{cases}
	\qquad\qquad
	\begin{cases}
	\partial_t u - \Delta u = -\chi_{\{u >0\}} \\
	u \geq 0, 
	\end{cases}
	\]
where $\varphi:\R^n \to \R$ is a smooth bounded obstacle satisfying $\Delta \varphi = -1$. Under suitable regularity assumptions, the second is obtained by the first by setting $u := v - \varphi$.

In this section, we show existence of solutions to the fully nonlinear analogue of the first formulation above, as stated in the following proposition. The approach we follow is the so-called ``penalization method'', combined with suitable uniform estimates and a compactness argument (see for instance \cite[Section 1]{Lee1998} and \cite[Chapter 1]{PetShahUralt} for the elliptic framework).
%
%
%
%
\begin{prop}\label{thm:existence} Let $\Omega \subset \R^n$ be a bounded domain with smooth boundary, $T > 0$ and $Q := \Omega \times (0,T)$. Let $\tilde{F} : \mathcal{S} \to \R$ satisfying \eqref{eq:AssOnFIntro} (independent of $x$) and $\tilde{g},\varphi \in C^\infty(\overline{\Omega})$ be such that $\tilde{g} \geq \varphi$ in $\overline{\Omega}$. Then there exists a function $v$ continuous up to $\partial_pQ$ satisfying
\[
v \in C^{1,\alpha}_p \cap W^{2,p}_x \cap W_t^{1,p} \text{ locally in } Q
\]
for all $\alpha \in (0,1)$, all $p \in (1,\infty)$ and
\begin{equation}\label{eq:ParObProb2}
\begin{cases} 
\partial_t v - \tilde{F}(D^2v) \geq 0, \quad v \geq \varphi &\text{ a.e. in } Q \\
\partial_t v - \tilde{F}(D^2v) = 0                          &\text{ in } \{v > \varphi\} \\
v = \tilde{g}                                               &\text{ in } \partial_p Q.
\end{cases}
\end{equation}
Furthermore, the function $u := v - \varphi$ is as regular as $v$ and satisfies 
\begin{equation}\label{eq:ProblemIntroBis}
\begin{cases}
\partial_tu - F(D^2u,x) = f(x) \chi_{\{u > 0\}} &\text{a.e. in } Q \\
u \geq 0,                                       &\text{in } Q \\
u = g,                                &\text{in } \partial_p Q
\end{cases}
\end{equation}
where 
\[
F(M,x) := \tilde{F}(M + D^2\varphi(x)) - \tilde{F}(D^2\varphi(x)), \qquad f(x) := \tilde{F}(D^2\varphi(x)). 
\]
and $g := \tilde{g} - \varphi$.
\end{prop}
\begin{proof} For any $\vep \in (0,1)$, we set $\beta_\vep(z) := e^{-z/\vep}$ and we consider the solution $w_\vep$ to the penalized problem
\begin{equation}\label{eq:PenProb}
\begin{cases}
\partial_t w_\vep - \tilde{F}(D^2w_\vep) = \beta_\vep(w_\vep-\varphi)   &\text{ in } Q \\
w_\vep = \tilde{g}                                                             &\text{ in } \partial_p Q.
\end{cases}
\end{equation}
Viscosity solutions to \eqref{eq:PenProb} can be constructed by means of the Perron's method as follows: by the definition of $\beta_\vep$ and since $\tilde{g} \geq \varphi$, one can easily check that the functions
\[
w_-(x,t) := -\|\tilde{g}\|_{L^\infty(\Omega)} \qquad \text{ and } \qquad w_+(x,t) := \tilde{g}(x) + (1+C)t
\]
are (ordered) subsolutions and superolutions to problem \eqref{eq:PenProb}, respectively, where $C := \| F(D^2\tilde{g})\|_{L^\infty(\Omega)}$. Consequently, the ``least supersolution'' defined as
\[
w_\vep(x,t) := \inf \{w(x,t): w \text{ is a supersolution to } \eqref{eq:PenProb} \text{ and } w \geq \tilde{g} \text{ in } \partial_p Q \}
\]
is a viscosity solution to \eqref{eq:PenProb} satisfying $w_- \leq w_\vep \leq w_+$, see for instance \cite[Section 3]{IS12}. Furthermore, by the comparison principle, we easily see that $w_\vep$ is unique.

By the argument above, we have a family of viscosity solutions $\{w_\vep\}_{\vep \in (0,1)}$ to \eqref{eq:PenProb}, uniformly bounded in $L^\infty(Q)$. Next we show that 
\begin{equation}\label{eq:CompactnessEsts}
\{w_\vep\}_{\vep \in (0,1)} \text{ is uniformly bounded in } C^{1,\alpha}_p \cap W^{2,p}_x \cap W_t^{1,p} \text{ locally in } Q,
\end{equation}
for all $\alpha \in (0,1)$ and all $p \in (1,\infty)$.\footnote{Actually, since $\tilde{g} \in C^\infty(\overline{\Omega})$ and $\partial \Omega \in C^\infty$, we also have $w_\vep \in C^{1,\alpha}_p$ up to $\partial_p Q$, uniformly w.r.t. $\vep \in (0,1)$.} 
In light of the Sobolev estimates \cite[Theorem 5.6, Theorem 5.7]{Wang92} and the Schauder estimates \cite[Theorem 4.8, Theorem 4.12]{Wang92bis}, it is enough to show that the r.h.s. of the penalized equation is uniformly bounded. Specifically, we prove that
\begin{equation}\label{eq:UnifBoundExis}
0 \leq h_\vep \leq \max\{1,\|f\|_{L^\infty(\Omega)}\},  
\end{equation}
for all $\vep \in (0,1)$, where $h_\vep := \beta_\vep(w_\vep - \varphi)$ and $f := \tilde{F}(D^2\varphi)$. The bound from below directly follows from the definition of $\beta_\vep$. Now let $(x_0,t_0)$ be a maximum point of $h_\vep$ in $\overline{Q}$. If $(x_0,t_0) \in \partial_p Q$, then
\[
h_\vep(x_0,t_0) = \beta_\vep(w_\vep(x_0,t_0) - \varphi(x_0)) = \beta_\vep(\tilde{g}(x_0) - \varphi(x_0)) \leq \beta_\vep(0) = 1,
\]
since $\tilde{g} \geq \varphi$. If $(x_0,t_0) \in Q$ is a local maximum of $h_\vep$ then, since $\beta_\vep$ is decreasing, $(x_0,t_0)$ is a local minimum of $w_\vep - \varphi$ and so $\partial_t(w_\vep - \varphi)(x_0,t_0) = 0$ and $D^2(w_\vep - \varphi)(x_0,t_0) \geq 0$. Consequently, by ellipticity of $\tilde{F}$, we obtain
\[
\begin{aligned}
h_\vep(x_0,t_0) &= \partial_t(w_\vep - \varphi)(x_0,t_0) - \tilde{F}(D^2(w_\vep - \varphi)(x_0,t_0) + D^2\varphi(x_0)) \\
&\leq - \tilde{F}(D^2\varphi(x_0)) = -f(x_0) \leq \|f\|_{L^\infty(\Omega)}.
\end{aligned}
\]
A similar argument yields the same bound when $(x_0,t_0) \in \Omega\times\{t=T\}$ and \eqref{eq:UnifBoundExis} follows.

As mentioned above, \eqref{eq:UnifBoundExis} gives us \eqref{eq:CompactnessEsts} which, in turn, yields the existence of $v \in C^{1,\alpha}_p \cap W^{2,p}_x \cap W_t^{1,p}$ locally in $Q$ and a sequence $\vep_k \to 0$ such that
\[
w_{\vep_k} \to v \quad \text{ in } C^{1,\alpha}_p \cap W^{2,p}_x \cap W_t^{1,p} \text{ locally in } Q
\]
as $k \to +\infty$, for all $\alpha \in (0,1)$ and all $p \in (1,\infty)$.

We are left to show that $v$ is a solution to \eqref{eq:ParObProb2}. The first inequality in \eqref{eq:ParObProb2} follows by passing to the limit in the viscosity sense and using that $\beta_\vep \geq 0$ (the differential inequality is also satisfied a.e. thanks to the Sobolev regularity of $v$). Further, we have $v \geq \varphi$ in $Q$. Indeed, if by contradiction there are $(x_0,t_0) \in Q$ and $\delta > 0$ such that $(v-\varphi)(x_0,t_0) = -\delta$, then $(w_{\vep_k}-\varphi)(x_0,t_0) \leq -\delta/2$ for large $k$'s and so
\[
h_{\vep_k}(x_0,t_0) \geq \beta_{\vep_k}(-\delta/2) = e^{\frac{\delta}{2{\vep_k}}} \to +\infty,
\]
as $k \to +\infty$, in contradiction with \eqref{eq:UnifBoundExis}. To see that $v$ satisfies the equation in $\{v > \varphi\}$, it is enough to check that
\[
h_{\vep_k} \to 0 \quad \text{ locally uniformly in } \{v > \varphi\},
\]
as $k \to +\infty$ (and pass to the limit in the viscosity sense, in compact subsets of $\{v > \varphi\}$). The proof of the limit above is straightforward: if $(x_0,t_0) \in \{v > \varphi\}$ with $(v - \varphi)(x_0,t_0) := \delta > 0$, uniform convergence shows that $(w_{\vep_k} - \varphi)(x_0,t_0) \geq \delta/2$ and thus 
\[
0 \leq h_{\vep_k}(x_0,t_0) \leq \beta_{\vep_k}(\delta/2) = e^{-\frac{\delta}{2{\vep_k}}} \to 0,
\]
as $k \to +\infty$. Finally, notice that $v = \tilde{g}$ in $\partial_p Q$ by uniform convergence up to $\partial_pQ$ and since $w_\vep = \tilde{g}$, for all $\vep \in (0,1)$.

To complete the proof, we notice that, since $\varphi \in C^\infty$, $u := v -\varphi$ is as regular as $v$. In particular, $u \in W^{2,p}_x \cap W_t^{1,p}$ locally in $Q$. This implies $\partial_t u = \nabla u = D^2u = 0$ a.e. in $\{u = 0\}$ (by Rademarcher's theorem) which, in turn, shows that $u$ is a solution to \eqref{eq:ProblemIntroBis}.
\end{proof}

%
%
%
%
%
\section{$C^{1,1}_x\cap C^{0,1}_t$ regularity and non-degeneracy}\label{sec:optimalRegularity}
In this section we establish the optimal $C^{1,1}_x$ spatial regularity of solutions to \eqref{eq:ProblemIntro} and the almost optimal $C^{0,1}_t$ time regularity.
The main result of this section is the following theorem. 
\begin{thm}\label{thm:OptimalReg} Let $u \in \mathcal{P}_1(K)$. Then $u \in C_x^{1,1}\cap C_t^{0,1}(Q_{1/2})$ and, further, there exists $C > 0$ depending only on $n$, $\lambda$, $\Lambda$ and $K$ such that
\begin{equation}\label{eq:OptimalReg}
\|\partial_t u\|_{L^\infty(Q_{1/2})} + \| D^2 u \|_{L^\infty(Q_{1/2})} \leq C.
\end{equation} 
\end{thm}
It is important to mention that the validity of the statement above is known, even in a more general framework: see \cite{FigShahBis,PetShah,Shah} and also \cite{FigShah,Lee1998} for the elliptic framework. Respect to these works, our approach heavily exploits the non-negativity and time-monotonicity of solutions, and it is equivalent to establish the optimal growth of solutions near the free boundary (see Lemma \ref{lem:OptimalGrowth}). Here we present two independent and new proofs: the first combines a special Harnack inequality with a comparison argument, while the second consists in a blow-up argument and has a perturbative flavour.

We begin with the following technical lemma.
\begin{lem}\label{lem:parabolaHarnack}

   Let $u \in \mathcal{P}_r(x_0,t_0;K)$ and let $\delta > 0$. Then there exists $C > 0$ depending only on $n$, $\lambda$, $\Lambda$ and $\delta$ such that
\begin{equation}\label{eq:HarnackIneqParbola}
\sup_{P_r^\delta(x_0,t_0)} u \leq C \big( u(x_0,t_0) + r^2\|f\|_{L^\infty(Q^-_r(x_0,t_0))} \big),
\end{equation}
where
\begin{equation}\label{eq:ParabolaDeltaDef}
P_r^\delta(x_0,t_0) := \{(x,t) \in Q_{r/2}^-(x_0,t_0): t - t_0 < -\delta|x-x_0|^2\}.
\end{equation}

\end{lem}
\begin{proof}   By Remark \ref{rem:Scaling}, it is enough to prove the statement for $r = 1$ and $(x_0,t_0) = (0,0)$. Let us set $P^\delta := P_1^\delta(0,0)$ and consider 
\[
D_{r/2}^+ := B_{r/2}\times(-\tfrac{1}{4}r^2,0) \quad \text{ and } \quad D_{r/2}^- := B_{r/2}\times(-\tfrac{7}{4}r^2,-\tfrac{3}{2}r^2),
\]
for $r \in I := (0,\sqrt{7}/7)$ (with $r$ in this range we automatically have $D_{r/2}^- \subset Q_{1/2}^-$). We first notice that applying the Harnack's inequality Theorem \ref{lem:Harnack} at every scale $r \in I$, we obtain
\[
\sup_{D_{r/2}^-} u \leq C \bigg( \inf_{D_{r/2}^+} u + r^2 \| f \|_{L^\infty(Q_1^-)} \bigg) \leq C \big( u(0,0) + \| f \|_{L^\infty(Q_1^-)} \big), \qquad \forall r \in I.
\]
By the arbitrariness of $r \in I$, it follows
\begin{equation}\label{eq:HarnackIneqParProof}
\sup_{A_0} u \leq C \big( u(0,0) + \| f \|_{L^\infty(Q_1^-)} \big), \qquad A_0:= P^7 = \big\{(x,t) \in Q_{1/2}^-: t < -7|x|^2 \big\}.
\end{equation}
We iterate this inequality as follows. Given a point $(y,\tau) \in Q_{1/2}^-$, we consider the set
\[
\tilde{A}_{y,\tau} := \big\{(x,t)\in Q_{1/2}^-: t - \tau < -7|x-y|^2 \big\}
\]
and we define inductively the family 
\[
\begin{cases}
R_0 := A_0 \\
R_k := \overline{A}_k \setminus A_{k-1}, \quad k \in \N\setminus\{0\}
\end{cases}
\qquad \text{where} \quad A_k := \bigcup_{(y,\tau) \in \partial A_{k-1}} \tilde{A}_{y,\tau}.
\]
By construction $\{R_k\}_{k\in\N}$ is a partition of $Q_{1/2}^-$. Now, we fix $\delta \in (0,1)$, pick $k_\delta \in \N$ such that
\[
P^\delta \subset \bigcup_{k=0}^{k_\delta} R_k
\]
and, for an arbitrary $(x,t) \in P^\delta$, we take $k \in \{0,\dots,k_\delta\}$ such that $(x,t) \in R_k$. Consequently, by construction and \eqref{eq:HarnackIneqParProof}, we have
\[
\begin{aligned}
u(x,t) &\leq \sup_{R_k} u \leq C \big( \sup_{R_{k-1}} u + \| f \|_{L^\infty(Q_1^-)} \big) \leq C^k \bigg( \sup_{P_0} u + \| f \|_{L^\infty(Q_1^-)} \sum_{j=0}^{k-1}C^{-j} \bigg) \\
&\leq C^{k+1} \big( u(0,0) + \tfrac{C}{C-1}\| f \|_{L^\infty(Q_1^-)} \big) \leq 2C^{k_\delta+1} \big( u(0,0) + \| f \|_{L^\infty(Q_1^-)} \big).
\end{aligned}
\]
The thesis follows by the arbitrariness of $(x,t) \in P^\delta$.
 \end{proof}
Next we establish the optimal growth control of solutions near the free boundary.
\begin{lem}\label{lem:OptimalGrowth} (Optimal growth)
  Let $u \in \mathcal{P}_1(K)$. Then there exists $C > 0$ depending only on $n$, $\lambda$, $\Lambda$ and $K$ such that 
\begin{equation}\label{eq:OptimalGrowth}
||u||_{L^\infty(Q_r(x_0,t_0))}\leq C r^2,
\end{equation}
for every $(x_0,t_0) \in \{u=0\}\cap Q_{1/2}$ and every $r \in (0,\tfrac{1}{2})$. 
 
\end{lem}
\begin{proof}[First proof of Lemma \ref{lem:OptimalGrowth}.]   Let us fix $(x_0,t_0) \in \{u=0\}\cap Q_{1/2}$, and set $b := \tfrac{1}{2\Lambda n}$ and $\delta := \tfrac{b}{2}$. We first notice that by \eqref{eq:HarnackIneqParbola}, there is $C_0 > 0$ depending on $n$, $\lambda$ and $\Lambda$ such that
\begin{equation}\label{eq:OptBdPardel}
\sup_{P_r^\delta(x_0,t_0)} u \leq C_0 \big( u(x_0,t_0) + r^2\|f\|_{L^\infty(Q^-_r(x_0,t_0))} \big) \leq C_0 K r^2,
\end{equation}
for every $r \in (0,\tfrac{1}{2})$, where $P_r^\delta(x_0,t_0)$ is defined in \eqref{eq:ParabolaDeltaDef}. In particular, it follows
\begin{equation}\label{eq:OptBdBoundaryPardel}
\sup_{x \in \partial B_r(x_0),t=t_0-\delta r^2} u(x,t) \leq C_0 K r^2,
\end{equation}
for every $r \in (0,\tfrac{1}{2})$.

Now, in view of \eqref{eq:OptBdPardel}, it is enough to focus on the set $Q_r(x_0,t_0) \setminus P_r^\delta(x_0,t_0)$. To prove the optimal growth on such set we proceed with a comparison argument as follows. Let us define
\[
v(x,t) := a (t-t_0 + b|x-x_0|^2), \qquad \quad a := \max\{2C_0,8\} \cdot \tfrac{K}{b}.  
\]
By uniform ellipticity, the assumption $F(O,\cdot) \equiv 0$ and the definition of $b$, we see that
\[
\partial_t v - F(D^2v,x) = a - F(2nab I, x) \geq a (1 - 2\Lambda nb \|I\|) = 0 \quad \text{ in } Q_1.
\]
Further, by definition of $v$ and $\delta$, we have
\[
v(x,t)|_{t = t_0 -\delta|x-x_0|^2} = a(b -\delta)|x-x_0|^2 = \tfrac{ab}{2}|x-x_0|^2,
\]
and so
\begin{equation}\label{eq:OptGrowBdBelowSuper}
v(x,t)|_{t = t_0 -\delta|x-x_0|^2} = \tfrac{ab}{2} r^2 \quad \text{ in } \partial B_r(x_0),
\end{equation}
for every $r \in (0,\tfrac{1}{2})$. On the other hand,
\[
v(x,t) \geq  a (-\tfrac{\delta}{4}  + \tfrac{b}{4} ) = \tfrac{ab}{8}    \quad \text{ in } \partial_p Q_{1/2}(x_0,t_0) \setminus P_{1/2}^\delta(x_0,t_0).
\]
At this point, combining \eqref{eq:OptBdBoundaryPardel} with \eqref{eq:OptGrowBdBelowSuper} and using that $\tfrac{ab}{2} \geq C_0 K$ by definition of $a$,  we obtain $v \geq u$ in $\partial P_{1/2}^\delta(x_0,t_0)\cap Q_{1/2}(x_0,t_0)$ while, since $\tfrac{ab}{8} \geq K$, we also have $v \geq u$ in $ \partial_p Q_{1/2}(x_0,t_0) \setminus P_{1/2}^\delta(x_0,t_0)$. Consequently, by the comparison principle, it follows
\[
u \leq v \quad \text{ in }  Q_{1/2}(x_0,t_0) \setminus P_{1/2}^\delta(x_0,t_0),
\]
and thus, since $v \leq C r^2$ in $Q_r(x_0,t_0) \setminus P_r^\delta(x_0,t_0)$ for some $C > 0$ (depending only on $n$, $\lambda$, $\Lambda$ and $K$), the bound in \eqref{eq:OptimalGrowth} is proved.
\end{proof}
\begin{proof}[Second proof of Lemma \ref{lem:OptimalGrowth}]   We argue by contradiction,  assuming the existence of sequences $\{F_k\}_{k\in\N}$ satisfying \eqref{eq:AssOnFIntro}, $\{f_k\}_{k\in\N} \in L^\infty(B_1)$ and $\{u_k\}_{k\in\N} \in \mathcal{P}_1(K)$ solutions to
\begin{equation}\label{eq:ProblemIntrok}
\begin{cases}
\partial_tu_k - F_k(D^2u_k,x) = f_k(x) \chi_{\{u_k > 0\}} &\text{in } Q_1 \\
u_k,\partial_t u_k \geq 0             &\text{in } Q_1,
\end{cases}
\end{equation}
but
\begin{equation}\label{eq:CA2}
\sup_{r \in (0,1)}r^{-2}||u_k||_{L^\infty(Q_r(x_k,t_k))} \geq k,
\end{equation}
for some sequence $\{(x_k,t_k)\}_{k\in\N} \subset \{u_k=0\}\cap Q_{1/2}$. To obtain a contradiction, we show that a suitable rescaled and renormalized subsequence of $\{u_k\}_{k\in\N}$ converge to a non-negative, non-trivial entire solution satisfying $u(0,0) = 0$ and growing at infinity less than a polynomial of degree 2, in contrast with the Liouville theorem.\footnote{The Liouville theorem for entire parabolic fully nonlinear equations is an immediate consequence of the $C^{2,\alpha}$ estimates proved in \cite[Theorem 4.13]{Wang92bis}}

In this spirit, we consider the monotone non-increasing function $\theta:(0,1) \to \R_+$, defined as
\[
\theta(r) = \sup_{k\in\N} \sup_{\rho \in (r,1)} \rho^{-2}||u_k||_{L^\infty(Q_\rho(x_k,t_k))}.
\]
By assumption, we have $|u_k| + |f_k| \leq K$ in $Q_1$ for every $k$. Combining this with \eqref{eq:CA2}, we easily deduce that $\theta(r) \leq K / r^2$ for every $r\in(0,1)$, with $\theta(r)\to\infty$ as $r \to 0$.

Now, due to \eqref{eq:CA2} we have that for every $j\in\N$, there is $r_j' \in (0,1)$ such that $\theta(r_j')>j$. Furthermore, by definition of $\theta$, it is not difficult to see that there exist $k_j \in \N$, $r_j>r_j'$ and $(x_{k_j},t_{k_j}) \in \{u_{k_j} = 0\} \cap Q_{1/2}$ such that
\[
\theta(r_j')\geq r_j^{-2}||u_{k_j}||_{L^\infty(Q_{r_j}(x_{k_j},t_{k_j}))}\geq \tfrac{\theta(r_j')}{2}.
\]
So, using the definition of $\theta$ again, its monotonicity and $\theta(r_j')>j$, we obtain
\begin{equation}\label{eq:Propthetam}
\begin{aligned}
\theta(r_j) &\geq r_j^{-2}||u_{k_j}||_{L^\infty(Q_{r_j}(x_{k_j},t_{k_j}))}\geq \tfrac{\theta(r_j)}{2} \\
\theta(r_j) &\geq \tfrac{j}{2},
\end{aligned}
\end{equation}
which, in particular, implies that $r_j \to 0$ as $j \to \infty$. Then, we define the blow-up sequence
\[
v_j(x,t) := \frac{1}{\theta(r_j) r_j^2}u_{k_j}(r_jx + x_{k_j},r_j^2t + t_{k_j}), \qquad (x,t)\in Q_{1/r_j}, \; j\in\N.
\]
Notice that $v_j$ is non-negative with $v_j(0,0) = 0$ and, by \eqref{eq:Propthetam}, uniformly non-degenerate: 
\begin{equation}\label{eq:NonDeg}
||v_j||_{L^\infty(Q_1)} \geq \tfrac{1}{2}, \quad \forall j \in \N.
\end{equation}
Moreover, thanks to \eqref{eq:Propthetam} again and the monotonicity of $\theta$, we have
\begin{equation}\label{eq:GrowthControl}
||v_j||_{L^\infty(Q_R)}=\frac{1}{\theta(r_j)r_j^2}||u_{k_j}||_{L^\infty(Q_{Rr_j}(x_{k_j},t_{k_j}))}\leq \frac{\theta(Rr_j)(Rr_j)^2}{\theta(r_j)r_j^2}\leq R^2,
\end{equation}
for all $1\leq R<1/r_j$ and, using the equation for $u_{k_j}$, we easily see that the function $v_j$ satisfies
\[
\partial_t v_j - \tilde{F}_j (D^2v_j, x) = \tilde{f}_j(x)\chi_{\{v_j > 0\}} \quad \text{ in }Q_{1/r_j},
\]
where 
\[
\begin{aligned}
\tilde{F}_j(M,x) &:= \frac{1}{\theta(r_j)} F_{k_j}(\theta(r_j)M,r_jx + x_{k_j}) \\ 
\tilde{f}_j(x) &:= \frac{1}{\theta(r_j)} f_{k_j}(r_jx + x_{k_j}).
\end{aligned}
\]
By definition, the sequence $\{\tilde{F}_j\}_{j\in\N}$ is made of functions satisfying \eqref{eq:AssOnFIntro}, with ellipticity constants $\lambda$ and $\Lambda$, while $\tilde{f}_j \to 0$ locally uniformly in $\R^{n+1}$ as $j \to +\infty$. This has two consequences: 

- First, up to passing to a subsequence, we have $(x_{k_j},t_{k_j}) \to (\tilde{x},\tilde{t})$ and $\tilde{F}_j \to \tilde{F}$ locally uniformly, for some $(\tilde{x},\tilde{t}) \in Q_{1/2}$ and some $\tilde{F}$ satisfying \eqref{eq:AssOnFIntro}.

- Second, by \cite[Theorem 4.8]{Wang92bis}, for every fixed $R \geq 1$ we have 
\[
||v_j||_{C^{1,\alpha}(Q_R)}\leq C(R), \quad \forall j \in \N.
\]
Combining these facts with \eqref{eq:NonDeg} and \eqref{eq:GrowthControl}, we deduce that, $v_j \to v$ locally uniformly in $\R^{n+1}$ (up to passing to another subsequence), for some continuous function $v$ satisfying
\[
\begin{cases}
\partial_t v- \tilde{F}(D^2v,\tilde{x}) = 0 \quad \text{ in }\R^{n+1}\\
v\geq0, \; v(0,0)=0  \\
||v||_{L^\infty(Q_R)}\leq R^2, \; \forall R \geq 1 \\
||v||_{L^\infty(Q_1)}\geq 1/2 ,
\end{cases}
\]
Now, we apply the Liouville theorem to conclude that $v$ is a polynomial of degree at most $2$ in space and $1$ in time. Further, by the maximum principle \cite[Corollary 3.20]{Wang92} applied in $Q_R$ (for arbitrary $R>0$), we deduce that $v = 0$ in $\R^n\times(-\infty,0]$ which, in turn, implies that $v\equiv 0$. This contradicts the fact that $||v||_{L^\infty(Q_1)}\geq 1/2$ and gives us \eqref{eq:OptimalGrowth}.
\end{proof}
Combining  the growth control result with interior estimates (see \cite[Theorem 1.1]{CK17} or \cite[Theorem 1.1]{Wang92bis}), we obtain the $L^\infty$-bound for $\partial_t u$ and $D^2 u$.
\begin{proof}[Proof of Theorem \ref{thm:OptimalReg}] Since $\partial_t u = \partial_{ij} u = 0$ a.e. in $\{u = 0\}$ for every $i,j \in \{1,\ldots,n\}$, it is enough to focus on points in $\{u > 0\}$. So, let us fix $(y,\tau) \in \{ u > 0 \} \cap Q_{1/2}$ and let 
\[
d := \sup \{ r > 0 : Q_r(y,\tau) \subset \{ u > 0 \} \cap Q_{1/2} \}.
\]
By Schauder estimates \cite[Theorem 4.8, Theorem 4.12]{Wang92bis}
and \eqref{eq:OptimalGrowth}, we have
\[
\begin{aligned}
\|\partial_t u\|_{L^\infty(Q_{d/2}(y,\tau))} + \| D^2 u \|_{L^\infty(Q_{d/2}(y,\tau))} &\leq \frac{C_0}{d^2} \big( \| u \|_{L^\infty(Q_d(y,\tau))} + d^2\| f \|_{L^\infty(Q_d(y,\tau))} \big) \\
&\leq \frac{C_0}{d^2} \big( C d^2 + d^2\| f \|_{L^\infty(Q_1)} \big) \\
&\leq C_0 (C + K),
\end{aligned}
\]
for some constants $C_0,C > 0$ depending only on $n$, $\lambda$, $\Lambda$ and $K$. In particular,
\[
|\partial_t u (y,\tau)| + | D^2 u (y,\tau)| \leq C_0 ( C + 1)
\]
and thus \eqref{eq:OptimalReg} follows thanks to the arbitrariness of $(y,\tau) \in \{ u > 0 \} \cap Q_{1/2}$.  
\end{proof}
%
%
%
%
%
%
%
%
%
%
We end the section with the following non-degeneracy property. 
\begin{lem}\label{lem:NonDegeneracy}
   Let $u \in \mathcal{P}_1(K)$. Then there exists $c>0$ depending only on $n$, $\Lambda$ and $c_\circ$ such that
\begin{equation}\label{eq:NonDegeneracy}
||u||_{L^\infty(Q^-_r(x_0,t_0))} \geq cr^2,
\end{equation}
for every $(x_0,t_0) \in \partial \{ u > 0\} \cap Q_{1/2}$ and every $r \in (0,\tfrac{1}{2})$.

\end{lem}
\begin{proof}
  Let us fix $(x_0,t_0) \in \partial \{u > 0\} \cap Q_{1/2}$, $r \in (0,1/2)$ and let 
\[
\{(x_k,t_k)\}_{k\in\N} \subset \{u > 0\} \cap Q_{1/2} \quad \text{ such that } \quad (x_k,t_k) \to (x_0,t_0) \;\text{ as } k \to +\infty.
\]
We set $c := c_\circ/(2\Lambda n + 1)$ and consider the sequence
\[
v_k(x,t) := u(x,t) - u(x_k,t_k) - c(|x-x_k|^2 - (t-t_k)), \quad k \in \N.
\]
Then
\[
\begin{aligned}
\partial_tv_k - F(D^2v_k,x) &= \partial_t u - F(D^2u -2cI,x) + c \\
&= \partial_t u - F(D^2u,x) + F(D^2u,x) - F(D^2u -2cI,x) + c \\
&= f(x)\chi_{\{u > 0\}} + F(D^2(u-c|x|^2) + 2cI,x) - F(D^2(u-c|x|^2),x) + c  \\
&\leq -c_\circ + c(2\Lambda n + 1) = 0 \quad \text{ in } \{ u > 0 \} \cap Q_r^-(x_k,t_k), 
\end{aligned}
\]
for every $k \in \N$. Further by definition, we have 
\[
v_k(x_k,t_k) = 0 \quad \text{ and } \quad v_k < 0 \; \text{ in } \; \partial\{ u > 0 \} \cap Q_r^-(x_k,t_k) 
\]
for every $k \in \N$ and thus, by the maximum principle (\cite[Proposition 4.34]{IS12}) it follows
\[
0 = v_k(x_k,t_k) \leq \sup_{Q_r^-(x_k,t_k)} v_k = \sup_{\partial_p Q_r^-(x_k,t_k)} v_k = \sup_{\partial_p Q_r^-(x_k,t_k)} u - u(x_k,t_k) - cr^2.
\]
In turn, this implies
\[
\sup_{Q_r^-(x_k,t_k)} u \geq u(x_k,t_k) + cr^2,
\]
for every $k \in \N$. Since $u(x_k,t_k) \to u(x_0,t_0) = 0$ as $k \to +\infty$ and $c$ is independent of $k$, we obtain \eqref{eq:NonDegeneracy} by passing to the limit as $k \to +\infty$.
\end{proof}
\begin{rem} 
The non-degeneracy property \eqref{eq:NonDegeneracy} gives us nontrivial information about the geometry of the free boundary: it excludes that free boundary points are parabolic interior for $\{u = 0\}$, in the sense that
\[
\{u > 0\} \cap Q_r^-(x_0,t_0) \neq \emptyset, \quad \forall r \in (0,1),
\]
where $(x_0,t_0) \in \partial\{u > 0\}$ is fixed (see \cite[Subsection 1.2]{CafPetSha04}).
\end{rem}
\begin{rem} Let $u \in \mathcal{P}_1(K)$ and $(x_0,t_0) \in \partial\{ u > 0\} \cap Q_{1/2}$. Combining the non-degeneracy estimate \eqref{eq:NonDegeneracy} with time-monotonicity $\partial_t u \geq 0$, we deduce that the function $u_{t_0} := u|_{t = t_0}$ satisfies
\[
||u_{t_0}||_{L^\infty(B_r(x_0))} \geq cr^2,
\]
where $c>0$ is as in \eqref{eq:NonDegeneracy} and depends only on $n$, $\Lambda$ and $c_\circ$. 
\end{rem}
%
%
%

%
%
%
%
%
\section{Semi-convexity and \texorpdfstring{$C_t^1$}{} estimates}\label{sec:semiConvexity}
The purpose of this section is to establish a semi-convexity estimate for solutions $u \in \mathcal{P}_1(K)$ and a log-continuity estimate for their time-derivatives $\partial_tu$, as stated in the following proposition. It is important to mention that for the semi-convexity estimates we require that the function $F$ is independent of the variable $x$ (which is enough for our purposes): in this way, thanks to convexity of $F$, the second derivatives of the solution become super-solutions to the linearised equation. The main result of this section is the following:
\begin{prop}\label{prop:semicnovexity}
    Let $u \in \mathcal{P}_1(K)$ with $(0,0) \in \partial \{ u > 0 \}$. Then there exist $\vep,C > 0$ depending only on $n$, $\lambda$, $\Lambda$ and $K$ such that 
\begin{equation}\label{eq:LogContinuityut}
\partial_t u \leq C \big| \log \big(|x| + \sqrt{|t|} \big) \big|^{-\varepsilon}  \quad \text{ in } Q_1.
\end{equation}
Furthermore, if the function $F$ in \eqref{eq:AssOnFIntro} is independent of $x$ and
\begin{equation}\label{eq:RegAssFfSemiConv}
\|f\|_{C^{1,1}(B_1)} \leq K,
\end{equation}
then
\begin{equation}\label{eq:LogSemiConvexity}
\partial_{ee} u \geq - C \big| \log \big(|x| + \sqrt{|t|} \big) \big|^{-\varepsilon}  \quad \text{ in } Q_1,
\end{equation}
for every $e \in \Ss^{n-1}$.
\end{prop}
This was already know for the Laplacian, see \cite{Caf78Bis} and \cite{CF79}. The proof of the above statement relies on the iterative use of the lemmas we establish below, which exploit the Weak Harnack inequality from Lemma \ref{lem:iteratedWeakHarnack}.
\begin{lem}\label{lem:semiconvexity}    Let $u \in \mathcal{P}_1(K)$ with $(0,0) \in \partial \{ u > 0 \}$. Assume that the function $F$ in \eqref{eq:AssOnFIntro} does not depend on $x$ and \eqref{eq:RegAssFfSemiConv} holds true. Then there exist $a, b, q > 0$ depending only on $n$, $\lambda$, $\Lambda$ and $K$ such that if
\[
\partial_{ee} u \geq -\gamma \quad \text{ in } Q_r,
\]
for some $r \in (0,1)$, $\gamma > 0$ and $e \in \Ss^{n-1}$, then
\begin{equation}\label{eq:ImprovSemiConvex}
\partial_{ee} u \geq -\gamma + a\gamma^q - br^2 \quad \text{ in } Q_{r/2}.
\end{equation}
\end{lem}
\begin{proof}   Let us fix $r,\gamma > 0$, $e \in \Ss^{n-1}$ and $(z,\tau)\in \{u>0\}\cap Q_{r/2}$. We set 
\[
d := \sup\{ \rho > 0: Q_\rho(z,\tau) \subset \{u > 0\} \} < \tfrac{r}{2},
\]
and we fix $(z_0,\tau_0) \in \partial \{ u>0 \} \cap Q_d(z,\tau)$. Notice that $(z_0,\tau_0)$ is always in the bottom of $\partial_p Q_d(z,\tau)$ since $\partial_t u \geq 0$. For $h \in [0,d]$, we consider the points
\[
(y,t) := (z_0 + \tfrac{h}{d}(z-z_0), \tau_0+h^2), 
\]
with $(y,t) = (z_0,\tau_0)$ for $h=0$, $(y,t) = (z,\tau)$ for $h = d$ and $|z_0 - y| = h$ (obviously, $t - \tau_0 = h^2$). Further, since $(\partial_{e})^2 = (\partial_{-e})^2$, we may choose $e$ such that $e\cdot (z-z_0)\geq 0$, i.e., $e$ points ``inwards'' the ball $B_d(z)$. Notice that by Theorem \ref{thm:OptimalReg} we have
\begin{equation}\label{eq:OptRegSemiConv}
u \leq Ch^2, \quad |\nabla u|\leq Ch \quad \text{ in } Q_{h/2}(y,t),
\end{equation}
for every $h \in [0,d]$ and $C > 0$ depending only on $n$, $\lambda$, $\Lambda$ and $K$.

Now, we define the set
\[
A_h := \{(x',t')\in Q_{h/2}(y,t);\text{ }x'\cdot e=0 \}.
\]
Notice that by construction $A_h$ is at least $\tfrac{h}{4}$ away from $\partial Q_d(z,\tau)$ (and $\partial \{ u > 0\}$). For every $(x_0,t_0) \in A_h$, if $\tilde{x} := x_0 + \tfrac{1}{8} \sqrt{hd}\, e$, we have
\[
0\leq u(\tilde{x},t_0 )  = u(x_0,t_0) + \nabla u (x_0,t_0) \cdot (\tilde{x} - x_0) + \int_{x_0}^{\tilde{x}} \int_{x_0}^x \partial_{ee} u , 
\]
and thus, by \eqref{eq:OptRegSemiConv} and the definition of $\tilde{x}$, we obtain	
\[
-C_0 h^{\frac{3}{2} }d^{\frac{1}{2}}\leq \int_{x_0}^{\tilde{x}} \int_{x_0}^x \partial_{ee} u ,
\]
for some $C_0 \geq 2C$ depending only on $n$, $\lambda$, $\Lambda$ and $K$. This bound and the assumption $\partial_{ee}u + \gamma \geq 0$ in $Q_r$ yield
\begin{align*}
\gamma -C_0 h^{\frac{1}{2}} d^{-\frac{1}{2}}  &\leq  \frac{128}{hd} \int_{x_0}^{\tilde{x}} \int_{x_0}^x (\partial_{ee}u + \gamma)  \leq  \frac{128}{hd }\int_{x_0}^{\tilde{x}} \int_{x_0}^{\tilde{x}} (\partial_{ee}u + \gamma) \\
&= \frac{16}{\sqrt{hd}} \int_{x_0}^{\tilde{x}} (\partial_{ee}u + \gamma) = 2 \fint_{x_0}^{\tilde{x}} (\partial_{ee}u + \gamma),
\end{align*}
and thus, choosing $h := \big( \tfrac{\gamma}{2C_0} \big)^2 d$, it follows
\[
\fint_{x_0}^{\tilde{x}} (\partial_{ee}u + \gamma) \geq \tfrac{\gamma}{4}.
\]
Notice that the choice of $\gamma$ implies $|\tilde{x} - x_0| = \tfrac{C_0}{8\gamma}h \geq \tfrac{C_0}{8C}h \geq \tfrac{h}{4}$ ($C$ as in  \eqref{eq:OptRegSemiConv}). Consequently, by the arbitrariness of $x_0 \in B_{h/2}(y)$, we conclude 
\begin{equation}\label{eq:SemiConvBoundBelAverage}
\fint_{\mathcal{C}_{h}} (\partial_{ee}u + \gamma) \geq \tfrac{\gamma}{4},
\end{equation}
where $\mathcal{C}_{h}=\{(x'+s\frac{1}{8}\sqrt{hd}e,t');\text{ } (x',t')\in A_h, s\in(0,1) \}$ is a skew-cylinder.

In this last part, we set $v := \partial_{ee}u + \gamma$ and we exploit \eqref{eq:SemiConvBoundBelAverage} to prove \eqref{eq:ImprovSemiConvex}. Since $F$ is independent of $x$, it is not difficult to check that
\[
\partial_t v - DF(D^2u)D^2v = g + D^2(\partial_eu) D^2F(D^2u) D^2(\partial_eu) \quad \text{ in } \{ u > 0 \}\cap Q_1,
\]
where $g := \partial_{ee}f$ and thus, setting $a_{ij} := (DF(D^2u))_{ij}$ and recalling that $M \to F(M)$ is convex, we deduce  
\begin{equation}\label{eq:secondDerivsSupersol}
	\begin{cases}
	\partial_t v - a_{ij}(x,t) \partial_{ij} v \geq g  \quad &\text{ in } Q_d(z,\tau) \\
	v \geq  0 \quad &\text{ in } Q_d(z,\tau).
	\end{cases}
\end{equation}
Since the matrix $\{a_{ij}\}_{ij}$ is uniformly elliptic (with ellipticity constants $\lambda$ and $\Lambda$), we may apply Corollary \ref{cor:weakHarnack}  to obtain
\[
\bigg( \fint_{\mathcal{C}_{h}} v^p \bigg)^{\frac{1}{p}}\leq \frac{C}{\delta^m} \big( v(z,\tau) + d^2\|g\|_{L^\infty(Q_1)}\big),
\]
for some $C,m > 0$ and $p \in (0,1)$ depending only on $n$, $\lambda$ and $\Lambda$, and $\delta := \tfrac{h}{8d}$. On the other hand, by optimal regularity and using that $p \in (0,1)$, we have
\begin{equation}\label{eq:SemiConvL1LpBound}
\tfrac{\gamma}{4} \leq \fint_{\mathcal{C}_{h}} v \leq \| v \|_{L^\infty(\tilde{Q}_{h/8}(y_0,t_0))}^{1-p} \fint_{\mathcal{C}_{h}}  v^p \leq (2C)^{1-p} \fint_{\mathcal{C}_{h}} v^p,
\end{equation}
where $C$ is as in Theorem \ref{thm:OptimalReg}. Noticing that $\delta \sim \gamma^2$ and exploiting \eqref{eq:RegAssFfSemiConv}, we combine the last two inequalities to deduce 
\[
C \gamma^{\frac{1}{p} + 2m} \leq  v(z,\tau) + d^2 K, 
\]
for some new $C>0$ still depending only on $n$, $\lambda$, $\Lambda$ and $K$. The thesis follows by choosing $a := C$, $q:= \tfrac{1}{p} + 2m$, $b:=K$ and using the arbitrariness of $(z,\tau)$ in $\{u > 0\} \cap Q_{r/2}$.
\end{proof}
An analogous statement can be proved for the time derivative. 
\begin{lem}\label{lem:continuityOfTimeDerivative}
   Let $u \in \mathcal{P}_1(K)$ with $(0,0) \in \partial \{ u > 0 \}$. Then there exist $a, q > 0$ depending only on $n$, $\lambda$, $\Lambda$ and $K$ such that if
\[
\partial_t u \leq \gamma \quad \text{ in } Q_r,
\]
for some $r \in (0,1)$, $\gamma > 0$, then
\begin{equation}\label{eq:ImprovDecayut}
\partial_t u \leq \gamma - a\gamma^q  \quad \text{ in } Q_{r/2}.
\end{equation}
\end{lem}
\begin{proof} 
As in the proof of Lemma \ref{lem:semiconvexity}, we fix $r \in (0,1)$, $\gamma > 0$, $(z,\tau)\in \{u>0\}\cap Q_{r/2}$ and we define 
\[
d := \sup\{ \rho > 0: Q_\rho(z,\tau) \subset \{u > 0\} \} < \tfrac{r}{2}.
\]
Then we fix $(z_0,\tau_0) \in \partial \{ u>0 \} \cap Q_d(z,\tau)$ (belonging to the bottom of $Q_d(z,\tau)$) and we estimate $\partial_t u$ in a cylinder centred at $(z_0,\tau_0)$.

To do so, we fix $h \in (0,d)$ and, since $u(z_0,t) = 0$ for $t \leq \tau_0$, we infer by \eqref{eq:OptimalGrowth}
\[
\int_{\tau_0-3d^2}^{\tau_0-2d^2} \partial_tu(x,t)dt = u(x,\tau_0-2d^2)- u(x,\tau_0-3d^2)\leq Ch^2, \qquad \forall x \in B_h(z_0),
\]
where $C > 0$ is as in Lemma \ref{lem:OptimalGrowth}. Averaging over $D^-(z_0,\tau_0) := B_h(z_0) \times (\tau_0-3d^2,\tau_0-2d^2)$, it follows
\[
\fint_{D^-(z_0,\tau_0)} \partial_tu(x,t)dt \leq C \big( \tfrac{h}{d} \big)^2,
\]
and thus, re-writing such inequality in terms of $v := \gamma - \partial_t u$ and choosing $h := (\tfrac{\gamma}{2C})^{\frac{1}{2}}d$, we obtain
\[
\fint_{D^-(z_0,\tau_0)} v  \geq \tfrac{\gamma}{2}.
\]
If $p \in (0,1)$ is as in Lemma \ref{lem:iteratedWeakHarnack}, the same argument used in \eqref{eq:SemiConvL1LpBound} shows that
\begin{equation}\label{eq:LogContUtLpAverEst}
C \gamma^{\frac{1}{p}} \leq \bigg( \fint_{D^-(z_0,\tau_0)} v^p \bigg)^{\frac{1}{p}},
\end{equation}
for some new $C > 0$ depending only on $n$, $\lambda$, $\Lambda$ and $K$.

On the other hand, notice that
\[
\begin{cases}
\partial_t v - a_{ij}(x,t)\partial_{ij}v = 0 \quad &\text{ in } \{u > 0\} \cap Q_r \\
0 \leq v \leq \gamma                                  \quad &\text{ in } Q_r, 
\end{cases}
\]
while $v = \gamma$ in $\intrr(\{u = 0\}) \cap Q_r$. Consequently, 
\[
\begin{cases}
\partial_t v - a_{ij}(x,t)\partial_{ij}v \geq 0 \quad &\text{ in } Q_r \\
v \geq 0                                  \quad &\text{ in } Q_r. 
\end{cases}
\]
At this point, similar to the proof of Lemma \ref{lem:semiconvexity}, we may apply Lemma \ref{lem:iteratedWeakHarnack} with $\delta:= (\tfrac{\gamma}{2C})^{\frac{1}{2}}$ and so, thanks to \eqref{eq:LogContUtLpAverEst}, we deduce
\[
C \gamma^{\frac{1}{p} + \frac{m}{2}} \leq v(z,\tau) = \gamma - \partial_tu(z,\tau),
\]
for some new $C > 0$ and $m > 0$ depending only on $n$, $\lambda$, $\Lambda$ and $K$. Thanks to the arbitrariness of $(z,\tau) \in \{u > 0\} \cap Q_{r/2}$, this yields \eqref{eq:ImprovDecayut} with $a := C$ and $q := \tfrac{1}{p} + \tfrac{m}{2}$.
\end{proof}

Iterating the above estimates yields the logarithmic decay near the free boundary.

\begin{proof}[Proof of Proposition \ref{prop:semicnovexity}]  We first prove the existence of $C,\vep > 0$ and $k_0 \in \N$ depending only on $n$, $\lambda$, $\Lambda$ and $K$ such that the sequence $m_k := - \inf_{Q_{2^{-k}}} \partial_{ee}u$ satisfy
\begin{equation}\label{eq:DecaySupUee}
m_k \leq C k^{-\vep}, \quad \forall k \geq k_0.
\end{equation}
Let $a$, $b$ and $q$ as in Lemma \ref{lem:semiconvexity} and choose $C$, $\vep$ and $k_0$ such that 
\[
C \geq \big( \tfrac{2b}{a} \big)^{\frac{1}{q}}, \qquad \vep \leq \min\{\tfrac{1}{q},b\}, \qquad Ck_0^{-\vep} \leq \big( \tfrac{1}{aq} \big)^{\frac{1}{q-1}}. 
\]
We proceed by induction on $k \geq k_0$. The case $k = k_0$ follows by optimal regularity (see \eqref{eq:OptimalReg}) and the definition of $C$. Now, assume \eqref{eq:DecaySupUee} holds true for some $k > k_0$ and let us prove it for $k+1$. By \eqref{eq:ImprovSemiConvex} and the inductive assumption, we have  
\[
m_{k+1} \leq m_k - am_k^q + b2^{-2k} \leq C k^{-\vep} - aC^q k^{-q\vep} + b2^{-2k},
\]
where we have also used that the function $x \to x - ax^q$ is increasing in $\big(0,\sqrt[q-1]{1/(aq)}\big)$ and the definition of $k_0$. Further, since the function $x \to x^{-\vep}$ is convex, we have
\[
k^{-\varepsilon} - \varepsilon k^{-\varepsilon-1}\leq (k+1)^{-\varepsilon},
\]
and thus
\[
m_{k+1} \leq C (k+1)^{-\varepsilon} + b2^{-2k} + \varepsilon k^{-\vep-1} - aC^q k^{-q\vep}.
\]
At this point, we infer 
\[
b2^{-2k} + \varepsilon k^{-\vep-1} - aC^q k^{-q\vep} = \big( b2^{-2k} - \tfrac{a}{2}C^q k^{-q\vep} \big) + \big( \varepsilon k^{-\vep-1} - \tfrac{a}{2}C^q k^{-q\vep} \big) \leq 0,
\]
by the definition of $C$ and $\vep$ (notice that $\vep < \tfrac{1}{q}$ implies $\vep < \tfrac{1}{q-1}$), and \eqref{eq:DecaySupUee} follows.

Now, we show that \eqref{eq:DecaySupUee} yields \eqref{eq:LogSemiConvexity}. To see this, let us fix $(x,t) \in Q_1$ and let $k \in \N$ such that $(x,t) \in Q_{2^{-k}} \setminus Q_{2^{-k-1}}$, i.e., $2^{-k-1} \leq |x| + \sqrt{|t|} \leq 2^{-k}$. Thus if $k \geq k_0$, we have by \eqref{eq:DecaySupUee}
\[
-\partial_{ee}u(x,t) \leq C k^{-\vep} \leq C \big| \log \big(|x| + \sqrt{|t|} \big) \big|^{-\varepsilon},
\]
up to taking a larger $C > 0$. If $k \leq k_0$ and $C > 0$ is as in \eqref{eq:OptimalReg}, then
\[
-\partial_{ee}u(x,t) \leq C \leq (C k_0^{\vep}) k_0^{-\vep} \leq (C k_0^{\vep}) k^{-\vep} \leq C \big| \log \big(|x| + \sqrt{|t|} \big) \big|^{-\varepsilon}, 
\]
for a new constant $C > 0$, and \eqref{eq:LogSemiConvexity} follows.

The proof of \eqref{eq:LogContinuityut} is similar and exploits Lemma \ref{lem:continuityOfTimeDerivative} instead of Lemma \ref{lem:semiconvexity}.
\end{proof}
%
%
%
%
%

%
%
%
%
%
%
%
%
%
%
%
%
%

%
%
%
%
%
\section{Classification of blow-ups}\label{sec:blowUps}

In this section we classify blow-ups of solutions $u \in \mathcal{P}_1(K)$ at free boundary points $(x_0,t_0) \in \partial \{u > 0\}$ and we study the limit as $r \downarrow 0$ of the rescalings 
\begin{equation}\label{eq:RescalingBlowUp}
u_r(x,t) := \frac{u(x_0 + rx,t_0 + r^2t)}{r^2},
\end{equation}
introduced in \eqref{eq:ParNormRescaling} (from now on, we write $u_r$ instead of $u^{(x_0,t_0)}_r$ to keep the notations as simple as possible). As explained in the introduction, each of such rescaling satisfies \eqref{eq:RescalingEqn}. Consequently, by \eqref{eq:OptimalReg} and the Arzel\'a-Ascoli theorem, there is a sequence $r_k \downarrow 0$ and a function $u_0:\R^{n+1}\to\R$ (the blow-up of $u$ at $(x_0,t_0)$) such that
\begin{equation}\label{eq:CompactnessBlowUp}
u_{r_k} \to u_0 \quad \text{ in } C^{1,\alpha}_p\quad \text{locally in } \R^{n+1},
\end{equation}
as $k \to \infty$, for every $\alpha \in (0,1)$. Further, writing \eqref{eq:LogSemiConvexity} and \eqref{eq:LogContinuityut} in terms of $u_{r_k}$ and passing to the limit as $k \to \infty$, we immediately see that $\partial_t u_0 = 0$ and $\partial_{ee} u_0 \geq 0$ for every $e \in \Ss^{n-1}$. Finally, by stability of viscosity solutions under uniform limits (see \cite[Proposition 3.11]{IS12}), we deduce that
\begin{equation}\label{eq:BlowUpEqn}
\begin{cases}
u_0 \in C^{1,1}_{loc}(\R^n), \;\; u_0 \not\equiv 0, \;\; u_0 \geq 0 \\
0 \in \partial\{u_0 > 0\} \\ 
u_0 \text{ is convex} \\
-F(D^2u_0,x_0) = f(x_0) \chi_{\{u_0 > 0\}} \quad \text{in } \R^n.
\end{cases}
\end{equation}
Notice that the first two properties are direct consequences of optimal regularity \eqref{eq:OptimalReg} and non-degeneracy \eqref{eq:NonDegeneracy}, respectively. In what follows, we will always assume that any blow-up of $u$ at $(x_0,t_0)$ satisfies  problem \eqref{eq:BlowUpEqn}.

There are two different behaviours of blow-ups -- either the contact set $\{u_0=0\}$ has empty interior or not -- which lead to a very different behaviour of the solution near free boundary points. We formalise this in the following definition.
\begin{defi}\label{def:RegSingPoints}
Let $u \in \mathcal{P}_1(K)$ and $(x_0,t_0)\in \partial \{u>0\}$. We say that $(x_0,t_0)$ is a regular free boundary point, if there exists a blow-up at $(x_0,t_0)$ whose contact set has non-empty interior. That is, there exist a sequence $r_k\downarrow 0$ and a solution $u_0$ to \eqref{eq:BlowUpEqn}, such that \eqref{eq:CompactnessBlowUp} holds true and $\{u_0=0\}$ has non-empty interior. The set of regular free boundary points will be denoted with $\Reg(u)$.

We denote with $\Sigma(u) := \partial\{u > 0\} \setminus \Reg(u)$ the set of singular free boundary points. By definition, if $(x_0,t_0) \in \Sigma(u)$, then any blow-up of $u$ at $(x_0,t_0)$ has contact set with empty interior.
\end{defi}

We first turn our attention towards the blow-ups near regular points. We proceed in the spirit of \cite[Lemma 7]{Caf98}, by using the convexity of blow-ups and the fact that the contact set has non-empty interior: these two properties yield the Lipschitz regularity of $\Reg(u_0)$. 
\begin{lem}\label{lem:EstimatesGlobalConvexSolutions} 
    Let $w$ be a solution to \eqref{eq:BlowUpEqn} and assume that $B_\rho(-\tau e_n)\subset \{w=0\}$ for some $\rho>0$, $\tau \in(0,1)$. Then the following assertions hold true.

\medskip

(i) There exists $c_0 \in (0,1)$ depending only on $\rho$, such that $\partial_\sigma w \geq 0$ in $B_{\rho/2}$, for all $\sigma \in \Ss^{n-1}$ with $\sigma_n > 1 - c_0$.

\smallskip
		
(ii) There exists a Lipschitz function $g$ such that 
\[
\{w>0\} \cap B_{\rho/2} = \{ (x',x_n) \in B_{\rho/2}: x_n > g(x') \}.
\]
Further, the Lipschitz norm of $g$ is  bounded by a constant depending only on $\rho$.
	
\smallskip	
		
(iii) Let $c_0 \in (0,1)$ and $g$ as above, and define $d(x) := x_n - g(x')$, $x \in \{w>0\} \cap B_\rho$. Then there exists $c > 0$ depending only on $n$, $\lambda$, $\Lambda$, $c_\circ$ as in \eqref{eq:AssOnfIntro} and $\rho$ such that
\[
\partial_\sigma w  \geq c d \quad \text{ in } B_{\rho/2}\cap\{w>0\},
\]
for all $\sigma \in \Ss^{n-1}$ with $\sigma_n > 1 - \frac{c_0}{2}$.
\end{lem}
\begin{proof}

To prove (i) we notice that for every point $x\in B_{\rho/2}$ the line passing through $x$ with direction $\sigma\in \Ss^{n-1}$ intersects $B_\rho(-\tau e_n)$, whenever $\sigma_n>1 - c_0$ and $c_0$ is close enough to $1$. Indeed, let $y := x - \sigma \tau$ and compute
\[
|y - (-\tau e_n)| = |x + \tau (e_n - \sigma)| \leq |x| + |e_n - \sigma| \leq \rho,
\]
where we have used that $\tau \in (0,1)$, $x \in B_{\rho/2}$ and we have chosen $\sigma \in \Ss^{n-1}$ such that $|e_n - \sigma| \leq \tfrac{\rho}{2}$. In particular, it must be $\sigma_n > 1 - c_0$, for some $c_0 \in (0,1)$ depending only on $\rho$. We deduce that $\partial_{\sigma}w(x)\geq 0$ by convexity of $w$ and the fact that $\partial_{\sigma}w(y) = 0$.

\smallskip

To show (ii), we consider the level sets $\{w = \vep \}$, for $\vep > 0$ small. First, we notice that 
\begin{equation}\label{eq:StrictPosPosLevel}
\partial_n w > 0 \quad \text{ in } B_{\rho/2} \cap \{w = \vep \}.
\end{equation}
Indeed, let $v := \partial_n w$ and assume $v(x_0) = 0$ for some $x_0 \in B_{\rho/2} \cap \{w = \vep \}$. Then, differentiating the equation of $w$ and using part (i), we obtain
\[
\begin{cases}
a_{ij}(x)\partial_{ij}v = 0 \quad &\text{ in } B_r(x_0) \\
v \geq 0                    \quad &\text{ in } B_r(x_0), 
\end{cases}
\]
for some ball $B_r(x_0) \subset  B_{\rho/2} \cap \{w > 0 \}$ and some uniformly elliptic matrix $a_{ij} = a_{ij}(x)$ with ellipticity constants $\lambda$ and $\Lambda$. It thus follows that $x_0$ is a minimum for $v$ and so $v = 0$ in $B_r(x_0)$ by the strong maximum principle (\cite[Proposition 4.34]{IS12}). Consequently $w = \vep$ in $B_r(x_0)$, which is impossible since the function $f$ in the right-hand side of the equation of $w$ is strictly negative by \eqref{eq:AssOnfIntro} and $F(O,\cdot) = 0$ by \eqref{eq:AssOnFIntro}.

Now, in light of \eqref{eq:StrictPosPosLevel}, we may apply the Implicit Function Theorem to deduce the existence of a function $h$ such that $(x',x_n) \in B_{\rho/2} \cap \{w = \vep\}$ if and only if $x_n = h(x',\vep)$, with $\partial_\vep h > 0$. At this point, by monotonicity, we set
\[
g(x') := \inf_{\vep > 0} h(x',\vep) = \lim_{\vep \to 0} h(x',\vep), \qquad    x' \in  B_{\rho/2} \cap \{x_n = 0\},
\]
and thus, by definition, $(x',x_n) \in B_{\rho/2} \cap \{w > 0\}$ if and only if $x_n > g(x')$.

To complete part (ii), we are left to prove that $g$ is a Lipschitz function with Lipschitz norm depending only on $\rho$. To do so, given $x \in B_{\rho/2} \cap \partial\{ w > 0 \}$, we consider the cone $\mathcal{C}_{x,\rho}$ with vertex at $x$ and opening $\vartheta_{x,\rho} \in (0,\pi/2)$. The number $\vartheta_{x,\rho}$ is the smallest opening such that $B_\rho(-\tau e_n) \subset \mathcal{C}_{x,\rho}$. By convexity, we know that the ``lower'' part of the cone $\mathcal{C}_{x,\rho}^-$ is fully contained in $\{ w = 0 \}$, while the ``upper'' part $\mathcal{C}_{x,\rho}^+ \subset \{ w > 0 \}$. This implies that $g$ is Lipschitz. To prove that the Lipschitz norm does not depend on the point, it is enough to notice that $\vartheta_{x,\rho} \geq \vartheta_\rho$ for all $x \in \overline{B}_{\rho/2}$, where    
\[
\vartheta_\rho := \inf \{ \vartheta_{x,\rho}: x \in B_{\rho/2}\} > 0.
\]

Let us turn to point (iii) and establish the inequality for $\sigma = e_n$. Let $x \in B_{\rho/2}\cap\{w>0\}$ be fixed and denote $d := d(x) = x_n - g(x')$. By non-degeneracy (see \eqref{eq:NonDegeneracy}), it holds
\[
\sup_{B_{c_0d/2}(x',g(x'))} w \geq c \left(\frac{c_0}{2}d\right)^2,
\]
for some $c > 0$ depending only on $n$, $\lambda$, $\Lambda$, $c_\circ$ as in \eqref{eq:AssOnfIntro} and $\rho$. Further, by part (i), $w$ is non-decreasing in all the directions $\sigma = (\sigma',\sigma_n)$ satisfying $|\sigma'| \leq c_0$ and so
\[
w(x) \geq w(y) \geq c d^2
\]
for some new $c > 0$ depending also on $\rho$, where $y$ is any point in $\overline{B}_{c_0d/2}(x',g(x'))$ where the above supremum is attained. Consequently, since $(x',g(x'))$ is a free boundary point, we have
\[
\int_{g(x')}^{x_n}\partial_n w(x',\xi) d\xi = w(x) \geq c d^2,
\]
and hence, by the mean value theorem, there must be a point $y$ in the segment $(x',f(x'))$ and $x$ such that
\[
\partial_n w(y)\geq c d.
\]
Exploiting the convexity of $w$ again (or the monotonicity of $\partial_n w$), we deduce $\partial_nw(x) \geq \partial_nw(y)\geq c d$, and the case $\sigma = e_n$ follows. To deduce the claim for all $\sigma \in \Ss^{n-1}$ satisfying $\sigma_n > 1 - \tfrac{c_0}{2}$ as in the statement, it suffices to write $\sigma = a e_n + \nu$, where $a > \tfrac{c_0}{2}$ and $\nu \in \Ss^{n-1}$ with $\nu_n > 1 - c_0$, and exploit part (i) to deduce
\[
\partial_\sigma w = a \partial_n w + \partial_\nu w \geq \tfrac{c_0}{2} cd + 0 = cd,
\]
for some new $c > 0$ depending only on $n$, $\lambda$, $\Lambda$, $c_\circ$ as in \eqref{eq:AssOnfIntro} and $\rho$.
\end{proof}
Through a finer analysis of partial derivatives of solutions to \eqref{eq:BlowUpEqn}, we deduce that the blow-ups at regular points must be one-dimensional. 
\begin{lem}\label{lem:halfParabola}
Let $w$ be a solution to \eqref{eq:BlowUpEqn}. Assume that $\{w = 0\}$ contains a half-cone with non-empty interior and vertex at $0$. Then, there are $e \in \Ss^{n-1}$ and $a > 0$ such that 
\begin{equation}\label{eq:BlowUpLimitReg}
w(x) = a \, (e \cdot x)_+^2.
\end{equation}
Furthermore, $\bar{a} \leq a \leq \bar{b}$ for some $0 < \bar{a} \leq \bar{b}$ depending only on $n$, $\lambda$, $\Lambda$ and $c_\circ$ as in \eqref{eq:AssOnfIntro}.
\end{lem}
\begin{proof}     Let $\mathcal{C} \subseteq \{w = 0\}$ be the half-cone having non-empty interior. Up to a rotation of the coordinate system, we may assume $\mathcal{C} = \{ (x',x_n) \in \R^n : x_n < - \tan \vartheta |x'|\}$, for some $\vartheta \in [0,\pi/2)$.
For any $\sigma \in \Ss^{n-1}$ with $-\sigma \in \mathcal{C}$, we have that $v_\sigma := \partial_\sigma w \geq 0$ in $\R^n$,  as in $(i)$ from Lemma \ref{lem:EstimatesGlobalConvexSolutions}. Hence $\{w>0\}$ must be Lipschitz as in $(ii)$ from Lemma \ref{lem:EstimatesGlobalConvexSolutions}. Differentiating the equation of $w$, we deduce
\[
\begin{cases}
a_{ij}(x) \partial_{ij}v_\sigma = 0 \quad &\text{ in }  \{ w > 0 \} \\
v_\sigma = 0  \quad &\text{ in } \partial\{ w > 0 \},
\end{cases}
\]
for some uniformly elliptic matrix $\{a_{ij}\}_{ij}$ with ellipticity constants $\lambda$ and $\Lambda$. Now, define $\mu(\sigma) = \sup \{\mu\geq0:\text{ } \partial_\sigma w-\mu\partial_n w \geq 0\}$ and conclude that
\[
\partial_\sigma w = \mu(\sigma) \partial_n w \quad \text{ in } \R^n.
\]
Indeed, if there is a point where $\partial_\sigma w - \mu(\sigma)\partial_n w>0$ we could apply the Boundary Harnack theorem in Lipschitz domains \cite[Theorem 1.1]{DS20} to $v_\sigma-\mu(\sigma) v_{e_n}$ and $v_{e_n}$, to get that 
$$v_\sigma-\mu(\sigma) v_{e_n}\geq c_0 v_{e_n},$$
for some $c_0>0$, which contradicts the definition of $\mu(\sigma)$.
Since $\sigma$ can vary over an open subset of $\Ss^{n-1}$, we conclude that $\nabla w = b\partial_n w$ for some constant vector $b\in\R^n$. Hence the level sets of $w$ are hyperplanes perpendicular to $b$, that is, $w$ is one dimensional. Without loss of generality, $w = w(x_n)$. In particular, $\{ w > 0 \} =  \{ x_n > 0\}$. To complete the proof, we notice that 
\[
D^2w =  w'' M, \qquad F(w'' M,x_0) = - f(x_0) \quad \text{ in } \R_+,
\]
where $M$ is the $n \times n$ matrix with zero entries everywhere except at the $n\times n$ position, where the entry is $1$. If we show that $w'' = a$ in $\R_+$ for some $a > 0$, our statement follows since $w(0) = w'(0) = 0$. To see this, we fix $h,k \geq 0$ and we notice that by uniform ellipticity
\[
F(hM,x_0) - F(kM,x_0) = F(kM + (h-k)M,x_0) - F(kM,x_0) \geq \lambda |h-k| \|M\| ,
\]
that is, $h \not= k$ implies $F(hM,x_0) \not= F(kM,x_0)$. Consequently, since the r.h.s. of the equation of $w$ is constant, $w''$ must be constant as well. Finally, by uniform ellipticity, we have
\[
\Lambda a = \Lambda w'' \| M \| \geq F(w'' M,x_0) = - f(x_0) \geq c_\circ, 
\]
and thus $a \geq \bar{a}:= \tfrac{c_\circ}{\Lambda}$. Conversely, 
\[
K \geq - f(x_0) =  F(w'' M,x_0) \geq \lambda w'' \| M \| = \lambda a,
\]
which yields $a \leq \bar{b} := \tfrac{K}{\lambda}$.
\end{proof}
On the other hand, when the blow-up has contact set with empty interior, we show that it satisfies the equation in the whole space. Then, by the Liouville theorem, it has to be a quadratic polynomial.
\begin{lem}\label{lem:blow-upSingularPoint}
	Let $w$ be a solution to \eqref{eq:BlowUpEqn}. Assume that $\{w = 0\}$ has empty interior. Then
	\begin{equation}\label{eq:SingBlowUp}
	w(x) = x^T A x,
	\end{equation}
	for some $n \times n$ matrix $A \geq 0$.
\end{lem}
\begin{proof}
	Since $w$ is convex and its contact set has empty interior, it follows that $\{w = 0\}$ is contained in a hyperplane and, in particular, it has zero Lebesgue  measure. Hence, by \cite[Theorem 2.7]{B01}, $w$ satisfies $F(D^2w,x_0) = f(x_0)$ in $\R^n$ in the classical sense while, by optimal growth, $\|w \|_{L^\infty(B_R)}\leq C R^2$, for all $R>1$, where $C > 0$ is as in \eqref{eq:OptimalGrowth}. Then, by the Liouville theorem, we deduce that $w$ has to be a quadratic polynomial. Since $w(0) = |\nabla w(0)| = 0$, we conclude that $w(x)=x^TAx$ for some matrix $A$ . Since $w\geq 0$ also $A\geq 0$.
\end{proof}
A standard argument allows us to show that singular points have zero density (in the parabolic sense). 
\begin{lem}
	Let $u\in\mathcal{P}_1(K)$ and let $(x_0,t_0) \in \Sigma(u)$. Then:
	$$\lim_{r\to 0}\frac{\left|\{u=0\}\cap Q_r(x_0,t_0)\right|}{\left|Q_r\right|} = 0.$$
\end{lem}
\begin{proof}
	Without loss of the generality we can assume that $(x_0,t_0)=(0,0)$. Assume by contradiction that there is a sequence $r_k\to0$ such that
	$$\lim_{k\to\infty}\frac{\left|\{u=0\}\cap Q_{r_k}(x_0,t_0)\right|}{\left|Q_{r_k}\right|} =\theta> 0.$$
	Passing to a subsequence, we get $u_{r_{k_j}}\to u_0$ locally uniformly in $\R^{n+1}$ and $u_0$ satisfies \eqref{eq:BlowUpEqn}. Since $(0,0)$ is a singular point, $\{u_0=0\}$ has empty interior. Combined with the convexity of $u_0$, we conclude that it has to be contained in a hyperplane, say $\{x_1 = 0\}$.
	
	Since $u_0>0$ in $\{x_1\neq 0\}$ and $u_0$ is continuous, we have  for every $\delta>0$
	$$u_0\geq\varepsilon\quad \quad\text{in }\{|x_1|>\delta\}\cap Q_1,$$
	for some $\varepsilon>0.$
	
	Therefore by uniform convergence of $u_{r_{k_j}}$ to $u_0$ in $Q_1$ we have for $l$ large enough
	$$u_{r_{k_j}}\geq \frac{\varepsilon}{2}\quad \quad\text{in }\{|x_1|>\delta\}\cap Q_1.$$
	In particular, the contact set of $u_{r_{k_j}}$ is contained in $\{|x_1|\leq \delta\}\cap Q_1$, hence
	$$\frac{|\{u_{r_{k_j}}=0\}\cap Q_1|}{\left|Q_1\right|} \leq C\delta.$$
	Scaling back to $u$, we find that
	$$\frac{|\{u=0\}\cap Q_{r_{k_j}}|}{|Q_{r_{k_j}}|} \leq C\delta,$$
	for large $l$'s. Since $\delta>0$ was arbitrary, we conclude that
	$$\lim_{j\to\infty}\frac{|\{u=0\}\cap Q_{r_{k_j}}|}{|Q_{r_{k_j}}|} =0,$$
	which contradicts the contradiction assumption.
\end{proof}
\begin{rem}
The inverse implication of the above statement holds true as well, that is, any free boundary point with zero density is a singular point. The proof of this fact is an immediate consequence of the Lipschitz regularity of the free boundary near regular points (see Proposition \ref{prop:LipschitzRegularityOfFB}).
\end{rem}
%
%
%
%
%

%
%
%
%
\section{Analysis of regular points}\label{sec:regularPoints}

In this section we establish the $C^\infty$ regularity of the regular part of the free boundary, as the following theorem asserts.
\begin{thm}\label{thm:SmoothnessFB}
	
	Let $u \in \mathcal{P}_1(K)$ with $(0,0) \in \Reg(u)$ and assume that $f\in C^\infty.$ 
	Then, there exists $\varrho_0 > 0$ and a $C^\infty$ function $g$ such that 
	\[
	Q_{\varrho_0} \cap \{ u >0\}  = \{ (x',x_n,t) \in Q_{\varrho_0} : x_n > g(x',t) \},
	\]
	up to a rotation of the spatial coordinates. 
\end{thm}
A key step to transfer the information from the blow-up to the solution is the so called ``almost positivity'' lemma below. 
\begin{lem}\label{lem:AlmostNonnegativityLemma}
    Let $\varrho \in (0,1]$, $K > 0$ and let $\{ u_r \}_{r \in (0,1)}$ be a family of solutions to
\begin{equation}\label{eq:EqnBlowUpSeqPosLem}
\begin{cases}
\partial_t v - F(D^2v,rx) = f(rx)\chi_{\{v > 0\}} \quad  &\text{ in } Q_\varrho \\
v, \partial_t v \geq  0    \quad &\text{ in } Q_\varrho,
\end{cases}
\end{equation}
with $F$ and $f$ satisfying 
\begin{equation}\label{eq:RegAssFfLipschitz}
\|f\|_{C^{0,1}(B_1)}  + \|F\|_{C^{0,1}_x(\mathcal{S} \times B_1)} \leq K.
\end{equation}
Then there exist $\vep_0,r_0 \in (0,1)$ depending only on $n$, $\Lambda$, $K$, $c_\circ$ as in \eqref{eq:AssOnfIntro} and $\varrho$, such that if 
\[
\partial_\sigma u_r - \partial_t u_r - u_r \geq - \varepsilon \quad \text{ in } Q_\varrho,
\]
for some $r \in (0,r_0)$, $\vep \in (0,\vep_0)$ and $\sigma \in \Ss^{n-1}$, then 
\[
\partial_\sigma u_r - \partial_t u_r - u_r \geq 0 \quad \text{ in } Q_{\varrho/2}.
\]
\end{lem}
\begin{proof}
    By scaling (see Remark \ref{rem:Scaling}), we may assume $\varrho = 1$. The proof is based on a comparison argument as follows. Let us set $u := u_r$ and $a_{ij}(x,t) := (DF (D^2u,rx))_{ij}$. Since the function $M \to F(M,\cdot)$ is convex, we have 
\[
0 = F(O)\geq F(D^2u(x,t),rx) - a_{ij}(x,t) \partial_{ij}u(x,t), 
\]
and hence
\begin{equation}\label{eq:SubSoluPosLemma1}
\partial_t u - a_{ij}(x,t) \partial_{ij}u \leq  f(rx) \quad \text{ in } Q_1 \cap \{u>0\}.
\end{equation} 
Differentiating the equation in \eqref{eq:EqnBlowUpSeqPosLem} along the direction $\sigma$, we obtain that $v := \partial_\sigma u$ satisfies 
\begin{equation}\label{eq:SolvPosLemma}
\partial_t v - a_{ij}(x,t) \partial_{ij} v = r\partial_\sigma F(D^2u,rx) + r \partial_\sigma f(rx), 
\end{equation}
while, differentiating with respect to $t$ and setting $\tilde{v} := \partial_t u$, we see that
\begin{equation}\label{eq:SolvtildePosLemma}
\partial_t \tilde{v} - a_{ij}(x,t) \partial_{ij} \tilde{v} = 0.
\end{equation}
Now, let $(x_0,t_0) \in Q_{1/2}\cap\{u>0\}$ be arbitrarily fixed, set 
\[
r_0 := \min\left\{1, c_\circ / (2K) \right\},
\]
and define
\[  
w = v - \tilde{v} - u + \tfrac{c_\circ}{4} \left[ \tfrac{1}{2n\Lambda}|x-x_0|^2 - (t-t_0) \right]. 
\]
Then, by \eqref{eq:AssOnfIntro} and combining \eqref{eq:SubSoluPosLemma1}, \eqref{eq:SolvPosLemma} and \eqref{eq:SolvtildePosLemma}, we have
\begin{align*}
		\partial_t w - a_{ij}(x,t) \partial_{ij}w &\geq r \left[\partial_\sigma F(D^2u,rx) +\partial_\sigma f(rx)\right] - f(rx) - \tfrac{c_\circ}{2}\\
        &\geq - r \left( \| \nabla_x F \|_{L^\infty(\mathcal{S}\times B_1)} + \| \nabla f \|_{L^\infty(B_1)} \right)	- f(rx) - \tfrac{c_\circ}{2} \\
        & \geq - rK - f(rx) - \tfrac{c_\circ}{2} \\
        & \geq -c_0 + f(rx) \geq 0 \qquad \text{ in } Q_{1/2}\cap \{u>0\},
\end{align*}
for all $r \in (0,r_0)$ while, since $|\nabla u| = \partial_t u = u = 0$ on $\partial \{ u > 0\}$, it must be $w > 0$ in $Q_{1/4}^-(x_0,t_0) \cap \partial\{ u > 0\}$. Further,  on $\partial_p Q_{1/4}^-(x_0,t_0)$, we have by assumption $w \geq - \vep + \vep_0$, where $\vep_0 := \tfrac{c_\circ}{128n\Lambda}$ and thus the minimum principle (\cite[Proposition 4.34]{IS12}) yields 
\[
\inf_{Q_{1/4}^-(x_0,t_0) \cap \{ u > 0\} } w = \inf_{\partial_p (Q_{1/4}^-(x_0,t_0) \cap \{ u > 0\} )} w \geq 0,
\]
for all $\vep \in (0,\vep_0)$. The thesis follows by the arbitrariness of $(x_0,t_0) \in Q_{1/2}\cap\{u>0\}$.
\end{proof}
%
%
%
%
%
Applying this result to the partial derivatives of solutions near regular points, we deduce that such derivatives are nonnegative along the direction of growth of the blow-up. This yields the Lipschitz regularity of the regular part of the free boundary.
\begin{prop}\label{prop:LipschitzRegularityOfFB}
 
  Let $u \in \mathcal{P}_1(K)$ with $(0,0) \in \Reg(u)$, and assume that $f$ and $F$ satisfy \eqref{eq:RegAssFfLipschitz}. 
Then, there exists $\varrho_0 > 0$ and a Lipschitz function $g$ such that 
\[
Q_{\varrho_0} \cap \{ u >0 \}  = \{ (x',x_n,t) \in Q_{\varrho_0} : x_n > g(x',t) \},
\]
up to a rotation of the spatial coordinates. 
\end{prop}
\begin{proof}     Since $(x_0,t_0) := (0,0)$ is a regular free boundary point, there exist a sequence $r_k\downarrow 0$ and solution $u_0$ to \eqref{eq:BlowUpEqn} such that the rescalings $u_k := u_{r_k}$ defined in \eqref{eq:RescalingBlowUp} satisfy \eqref{eq:CompactnessBlowUp} as $k \to +\infty$ and $\{u_0=0\}$ has non-empty interior. Consequently, thanks to the invariance of the problem under spatial rotations, we may assume $B_\rho(-\tau e_n) \subset \{u_0=0\}$ for some $0 < \rho < \tau < 1$. Thus, by Lemma \ref{lem:EstimatesGlobalConvexSolutions} (part (iii)) and \eqref{eq:OptimalGrowth}, it follows
\[
\partial_\sigma u_0 - u_0 \geq c d - C d^2 \geq 0 \quad \text{ in } B_{\rho/2} \cap \{u_0>0\} \cap \{d < c/C\},
\]
where $c > 0$, $\sigma \in \Ss^{n-1}$ and $d = d(x)$ are as in Lemma \ref{lem:EstimatesGlobalConvexSolutions} and $C > 0$ as in Lemma \ref{lem:OptimalGrowth} (in particular, $\sigma_n > 1 - c_0$, for some $c_0 \in (0,1)$ small depending on $\rho$). Combining this with \eqref{eq:CompactnessBlowUp}, we deduce that, given any $\vep > 0$, there is $k_\vep $ such that
\[
\partial_\sigma u_k - u_k \geq - \tfrac{\vep}{2} \quad \text{ in } Q_\varrho,
\]
for all $k \geq k_\vep$, where $\varrho > 0$ is taken small enough depending on $\rho$, $c$ and $C$. Further, the rescaled version of \eqref{eq:LogContinuityut} in Proposition \ref{prop:semicnovexity} gives
\[
\partial_\sigma u_k - \partial_t u_k - u_k \geq -\varepsilon \quad  \text{ in } Q_\varrho,
\]
up to taking $k_\vep$ larger. At this point, we fix $\vep := \vep_0/2$ and $k_\vep$ (eventually larger) such that $r_k \in (0,r_0/2)$ for all $k \geq k_\vep$, where $\vep_0$ and $r_0$ are as in Lemma \ref{lem:AlmostNonnegativityLemma}. Since each $u_k$ satisfies \eqref{eq:EqnBlowUpSeqPosLem} in $Q_\varrho$ with $r = r_k$, we conclude
\[
\partial_\sigma u_k - \partial_t u_k - u_k \geq 0  \quad \text{ in } Q_{\varrho /2},
\]
by virtue of Lemma \ref{lem:AlmostNonnegativityLemma} and thus, scaling back to $u$, it follows
\begin{equation}\label{eq:LipSpaceTimeFinalIneq}
\partial_\sigma u - r_k \partial_t u \geq \tfrac{1}{r_k}u   \quad \text{ in } Q_{r_k \varrho /2}.
\end{equation}
To complete the proof, we fix $k = k_\vep$ (notice that $k_\vep$ depends only on $n$, $\Lambda$, $K$ and $c_\circ$ as in \eqref{eq:AssOnfIntro}) and set $\varrho_0 := \tfrac{r_k \varrho}{4}$, $\nu := (\sigma,-r_k)$. The above inequality implies $\partial_\nu u \geq 0$ in $Q_{2\rho_0}$, with $\partial_\nu u > 0$ in $Q_{2\rho_0} \cap \{ u > 0\}$: these facts, combined with the arbitrariness of $\sigma \in \Ss^{n-1}$ with $\sigma_n \in (1-c_0,1)$, allow to repeat the argument used in the proof of Lemma \ref{lem:EstimatesGlobalConvexSolutions} part (ii) and to complete the proof of our statement.  
\end{proof}
Once the regular part is proved to be locally Lipschitz, we go ahead with the analysis of blow-up sequences around regular points and we give a uniform rate of convergence to the blow-up in a neighbourhood of the free boundary point. This is an important step towards the $C^1_x$ regularity of the free boundary. Moreover we conclude that the set of regular points is open inside the free boundary.

\begin{lem}\label{lem:uniformConvergenceToBlowUp}
 
  Let $u \in \mathcal{P}_1(K)$ with $(0,0) \in \partial\{ u > 0\}$, $0 < \bar{a} < \bar{b}$ as in Lemma \ref{lem:halfParabola} and let $\{ u_r \}_{r \in (0,1)}$ be the family of rescalings \eqref{eq:RescalingBlowUp}. Then for every $\vartheta \in [0,\pi/2)$ and $\vep \in (0,1)$, there exists $r = r(\vartheta,\vep) \in (0,1)$ such that if
\[
Q_{1/r} \cap \{x_n \leq - \tan \vartheta \, (|x'|+\sqrt{|t|}) \} \subset Q_{1/r} \cap \{ u_r = 0 \},
\]
then there exist $e \in \Ss^{n-1}$ and $a \in [\bar{a},\bar{b}]$ such that
\[
\| u_r - a(e \cdot x)_+^2\|_{C_p^{1,\alpha}(Q_1)} \leq \varepsilon.
\]
\end{lem}
\begin{proof}
    Assume by contradiction there are $\theta \in [0,\pi/2)$ and $\vep \in (0,1)$ such that for every sequence $r_k\downarrow 0$ there holds 
\[
Q_{1/r_k} \cap \{x_n \leq - \tan \vartheta \, (|x'|+\sqrt{|t|}) \} \subset Q_{1/r_k} \cap \{ u_k = 0 \},
\]
but
\[
\| u_k - a (e \cdot x)_+^2||_{C_p^{1,\alpha}(Q_1)} \geq \varepsilon,
\]
for every $e \in \Ss^{n-1}$ and $a \geq \bar{a}$, where we have set $u_k := u_{r_k}$. Passing to a subsequence, we may assume $u_k \to u_0$ as $k \to +\infty$ in the sense of \eqref{eq:CompactnessBlowUp}, for some limit $u_0$ satisfying \eqref{eq:BlowUpEqn}. Furthermore, since 
\[
u_k = 0 \quad \text{ in } Q_{1/r_k} \cap \{x_n \leq - \tan \vartheta \, (|x'|+\sqrt{|t|}) \}
\]
for all $k \in \N$, we obtain $u_0 = 0$ in $\{x_n \leq - \tan \vartheta \, |x'| \}$, that is, $\{x_n \leq - \tan \vartheta \, |x'| \} \subset \{u_0 = 0\}$. We may thus apply Lemma \ref{lem:halfParabola} to deduce that $u_0$ is of the form \eqref{eq:BlowUpLimitReg}, obtaining the desired contradiction.
\end{proof}
We next improve the spatial regularity of the free boundary near regular points and show that the second spatial derivatives are continuous up to the boundary there.
\begin{prop}\label{prop:C1ofFB}
   
Let $u \in \mathcal{P}_1(K)$ with $(0,0) \in \Reg(u)$, and assume that $f$ and $F$ satisfy \eqref{eq:RegAssFfLipschitz}. Then, there exists $\varrho_0 > 0$ and a $C^1_{x'} \cap C^{0,1}_t$ function $g$ such that 
\[
Q_{\varrho_0} \cap \{ u >0\}  = \{ (x',x_n,t) \in Q_{\varrho_0} : x_n > g(x',t) \},
\]
up to a rotation of the spatial coordinates. Furthermore, there exists a modulus of continuity $\omega: \R_+ \to \R_+$ such that
\[
\sup_{(x,t),(y,\tau) \in Q_r \cap \overline{\{ u > 0\}}} |D^2u(x,t) - D^2u(y,\tau)| \leq \omega(r)
\]
for all $r \in (0,\tfrac{\varrho_0}{2})$.
\end{prop}
\begin{proof}
    To establish the first part of the statement, we show the existence of $\varrho,r_0 > 0$ and a modulus of continuity $\omega$ (that is, $\omega: \R_+ \to \R_+$ with $\omega(r) \to 0$, as $r \downarrow 0$) such that for every $(x_0,t_0) \in Q_\varrho \cap \partial\{u>0\}$, there is $e \in \Ss^{n-1}$ (depending on $(x_0,t_0)$) such that
\begin{equation}\label{eq:BoundFBModCont}
Q_r(x_0,t_0) \cap \partial\{u>0\} \subset \{ (x,t): |e\cdot (x-x_0)|\leq \omega(r) r \},
\end{equation}
for all $r \in (0,r_0)$. This will be achieved in a couple of steps as follows.

First, we show that there exists  $\varrho > 0$ such that for every $\vep \in (0,1)$, there is $r \in (0,1)$ such that, for every $(x_0,t_0) \in Q_\varrho \cap \partial \{ u > 0 \}$ there holds
\begin{equation}\label{eq:convergenceToBlowUpNew}
\|u_r - a (e\cdot (x-x_0))_+^2 \|_{C_x^{1,\alpha}\cap C_t^1(Q_1)} \leq \vep,
\end{equation}
for some $a \in [\bar{a},\bar{b}]$ and $e \in \Ss^{n-1}$ depending on $(x_0,t_0)$, where $\bar{a}$ and $\bar{b}$ are as in Lemma \ref{lem:halfParabola}. This claim follows by noticing that, in light of Proposition \ref{prop:LipschitzRegularityOfFB}, we have
\[
\{ u >0\} \cap Q_{2\varrho} = \{ (x',x_n,t) \in Q_{2\varrho} : x_n > g(x',t) \},
\]
up to a  rotation of spatial coordinates, for some $\varrho > 0$ and a Lipschitz function $g$ (with Lipschitz norm depending only on $\varrho$). Consequently, there is $\vartheta \in [0,\pi/2)$ and $r = r(\vartheta,\vep) \in (0,1)$ satisfying
\[
Q_{1/r} \cap \{x_n \leq - \tan \vartheta \, (|x'|+\sqrt{|t|}) \} \subset Q_{1/r} \cap \{ u_r = 0 \}.
\]
Then \eqref{eq:convergenceToBlowUpNew} follows directly from Lemma \ref{lem:uniformConvergenceToBlowUp}.

Second, we exploit \eqref{eq:convergenceToBlowUpNew} to show
\begin{equation}\label{eq:BoundFBsqrtVep}
Q_r(x_0,t_0) \cap \partial\{u>0\} \subset \{ (x,t): |e\cdot (x-x_0)|\leq C \sqrt{\varepsilon}r \},
\end{equation}
for all $(x_0,t_0) \in Q_\varrho \cap \partial\{u>0\}$ and some $C > 0$ depending only on $n$, $\lambda$, $\Lambda$ and $c_\circ$ as in \eqref{eq:AssOnfIntro}. Once this inclusion is established, \eqref{eq:BoundFBModCont} easily follows: notice that $\omega$ is uniform in $Q_\varrho$ since the constant $C$ above is independent of $(x_0,t_0)$, and $r_0$ is the biggest $r \in (0,1)$ for which \eqref{eq:convergenceToBlowUpNew} holds true (take for instance $\vep = 1$).

To check \eqref{eq:BoundFBsqrtVep}, let us choose $C > 0$ such that
\[
C^2 > \max\{1/\bar{a},1/c \},
\]
where $c > 0$ is as in Lemma \ref{lem:NonDegeneracy}. Then, for every $(x,t)\in Q_r(x_0,t_0) \cap \{ e \cdot (x-x_0) > C \sqrt{\varepsilon}r \}$, we have by rescaling \eqref{eq:convergenceToBlowUpNew}
\[
u(x,t) \geq a [e \cdot (x-x_0)]^2 - \vep r^2 \geq \bar{a} C^2\vep r^2 -\vep r^2 = \vep r^2 (\bar{a} C^2 - 1) > 0,
\]
while, whenever $(x,t)\in Q_r(x_0,t_0) \cap \{ e \cdot (x-x_0) < - C \sqrt{\varepsilon}r \}$ and $(x,t) \in \partial\{ u > 0 \}$, we obtain by non-degeneracy and \eqref{eq:convergenceToBlowUpNew} again
\[
c C^2\vep r^2 \leq \sup_{Q_{C\sqrt{\vep}r}(x,t)} u = \sup_{Q_{C\sqrt{\vep}r}(x,t)} u - a [e \cdot (x-x_0)]_+^2 \leq \vep r^2,
\]
which is impossible and thus \eqref{eq:BoundFBsqrtVep} follows.

At this point, it suffices to notice that \eqref{eq:BoundFBModCont} guarantees that at each point in $Q_\varrho$ the graph of $g$ can be touched from below and from above by the functions
\[
x_n = \pm \, \omega \big(|x'| + \sqrt{|t|} \big) \big( |x'| + \sqrt{|t|} \big),
\]
up to a rotation and a translation. Consequently, $g$ is differentiable in a neighbourhood of $(0',0)$ and, since $\omega$ is uniform in $Q_\varrho$, it follows that $\nabla_{x'} g$ is continuous in a neighbourhood of $(0',0)$ with modulus of continuity $2\omega$. The fact that $g \in C^{0,1}_t$ is a direct consequence of Proposition \ref{prop:LipschitzRegularityOfFB} and the first part of our statement follows with $\varrho_0 := \varrho$.

Now, we show that $D^2u$ is continuous in $Q_{\varrho_0/2} \cap \overline{\{ u > 0\}}$. As above, we proceed in some steps. Let us set
\[
p^{(x_0,t_0)}(x) := a_0 [e_0 \cdot (x-x_0)]_+^2, \qquad\qquad q^{(x_0,t_0)}(x) := a_0 [e_0 \cdot (x-x_0)]^2,
\]
where $a_0 \in [\bar{a},\bar{b}]$ and $e_0 \in \Ss^{n-1}$ stand for $a(x_0,t_0)$ and $e(x_0,t_0)$, respectively. By \eqref{eq:convergenceToBlowUpNew} and the first part of the proof, we have
\[
\| u - p^{(x_0,t_0)} \|_{L^\infty(Q_r(x_0,t_0))} \leq \omega(r)r^2,
\]
for all $r \in (0,r_0)$. Furthermore,  
\[
\| q^{(x_0,t_0)} - p^{(x_0,t_0)} \|_{L^\infty(Q_r(x_0,t_0) \cap \{u > 0\})} \leq \omega(r)r^2,
\]
for all $r \in (0,r_0)$, taking eventually $r_0$ smaller. To see this, let us assume $(x_0,t_0) = (0,0)$ and $e_0 = e_n$ (which is always the case up to a rotation and a translation) and set $p:= p^{(x_0,t_0)}$, $q := q^{(x_0,t_0)}$. Then $p - q = 0$ in $x_n \geq 0$, whilst $p-q = -q$ in $x_n < 0$ and thus, for every $(x,t) \in Q_r \cap \{u > 0\}$, we have
\[
|p(x)-q(x)| \leq  q(x) = a_0 x_n^2 \leq \bar{b} g(x',t)^2  \leq \bar{b} \omega^2(r)r^2 \leq \omega(r) r^2, 
\]
where $g$ is the $C^1_x \cap C^{0,1}_t$ function parametrizing $Q_r \cap \partial\{u > 0\}$ as above and $\bar{b} > 0$ is as in Lemma \ref{eq:BlowUpLimitReg}. Passing to the supremum, our claim follows.

Combining the two estimates above, we obtain
\begin{equation}\label{eq:TaylorUOrder0}
\| u - q^{(x_0,t_0)} \|_{L^\infty(Q_r(x_0,t_0) \cap \{u > 0\})} \leq 2\omega(r)r^2,
\end{equation}
for all $r \in (0,r_0)$ and, in a similar way, we also conclude
\begin{equation}\label{eq:TaylorUOrder1}
\| \nabla u - \nabla q^{(x_0,t_0)} \|_{L^\infty(Q_r(x_0,t_0)\cap\{u>0\})}\leq 2\bar{b} \, \omega(r)r.
\end{equation}
In particular, \eqref{eq:TaylorUOrder0} and \eqref{eq:TaylorUOrder1} imply that $D^2u(x_0,t_0)$ exists and equals $D^2 q^{(x_0,t_0)} := A_0$, for some suitable $n \times n$ matrix $A_0$ depending on $a_0$ and $e_0$ (and thus, on $(x_0,t_0)$). Actually, the following quantitative bound holds true:
\begin{equation}\label{eq:TaylorUOrder2}
\| D^2 u - D^2 q^{(x_0,t_0)} \|_{L^\infty(Q_r(x_0,t_0)\cap\{u>0\})}\leq C \omega(r),
\end{equation}
for all $r \in (0,r_0)$ and some $C > 0$ depending only on $n$, $\lambda$ and $\Lambda$.

To complete the proof, we notice that, since $p^{(x_0,t_0)}$ satisfies \eqref{eq:BlowUpEqn} and is $1$-dimensional, it immediately follows that $q^{(x_0,t_0)}$ satisfies
\[
-F(D^2q^{(x_0,t_0)},x_0) = -F(A_0,x_0) = f(x_0).
\]
Consequently, the function $v := u - q^{(x_0,t_0)}$ is a solution to
\[
\partial_t v - \tilde{F}(D^2v,x) = \tilde{f}(x) \quad \text{ in } Q_r(x_0,t_0) \cap \{ u > 0 \},
\]
where $\tilde{F}(M,x) := F(M + A_0,x) - F(A_0,x)$ and $\tilde{f}(x) := f(x) - F(A_0,x)$.
Notice that $|F(A_0,x) - F(A_0,x_0)|\leq \tilde{\omega}(|x-x_0|)$, as well as $|f(x)-f(x_0)|\leq \tilde{\omega}(|x-x_0|)$ for some modulus of continuity $\tilde{\omega}.$ Since $\tilde{F}$ belongs to the class \eqref{eq:AssOnFIntro}  and $\| f \|_{C^{0,1}} \leq K$, we may combine the Schauder estimates \cite[Theorem 4.8]{Wang92bis} with \eqref{eq:TaylorUOrder0} to deduce 
\begin{equation}\label{eq:growthOfSecondDerivativesNew}
\begin{aligned}
\| D^2u - D^2q^{(x_0,t_0)} \|_{L^\infty(Q_{r/4}(z,s))} &\leq \frac{C}{r^2} \big( \| u - q^{(x_0,t_0)} \|_{L^\infty(Q_{r/2}(z,s))} + r^2 \| \tilde{f} \|_{L^\infty(Q_{r/2}(z,s))} \big) \\
&\leq \frac{C}{r^2} \big( \omega(r) r^2 + K r^2 \tilde{\omega}(r)\big) \\
&\leq (C + K) \max\{\omega(r),\tilde{\omega}(r)\},
\end{aligned}
\end{equation}
where $Q_{r/2}(z,s) \subset Q_r(x_0,t_0) \cap \{ u > 0 \}$ and $C$ depends only on $n$, $\lambda$ and $\Lambda$. Applying this twice on $Q_{r/2}(z,s) \subset Q_r(x_0,t_0) \cap Q_r(y_0,\tau_0) \cap \{ u > 0 \}$, where $(y_0,\tau_0) \in \partial \{ u > 0 \}$, it follows
\[
\| D^2q^{(x_0,t_0)} - D^2 q^{(y_0,\tau_0)} \|_{L^\infty(Q_{r/4}(z,s))} \leq  C \omega(r),
\]
for some new $C > 0$ depending only on $n$, $\lambda$, $\Lambda$ and $K$. Since $q$ is a polynomial of degree 2, this is equivalent to say
\begin{equation}\label{eq:ContSecondDerivativesFB}
d((x_0,t_0),(y_0,\tau_0)) \leq r   \qquad \Rightarrow \qquad   \| D^2q^{(x_0,t_0)} - D^2 q^{(y_0,\tau_0)} \|_\infty \leq  C \omega(r),
\end{equation}
for all $r > 0$, that is, the function $\partial \{ u > 0 \} \ni (x,t) \mapsto D^2q^{(x,t)}$ is continuous.

To complete the proof, let us fix $(x,t),(y,\tau) \in \overline{\{ u > 0\}}$, set $\rho := d((x,t),(y,\tau))$ and consider $(x_0,t_0),(y_0,\tau_0) \in \partial \{ u > 0 \}$ projections of $(x,t)$ and $(y,t)$ over $\partial \{u > 0 \}$. Further, let us set 
\[
d_{x,t} := \sup \{ r > 0 : Q_r(x,t) \subset \{ u > 0 \} \}, \qquad d_{y,\tau} := \sup \{ r > 0 : Q_r(y,\tau) \subset \{ u > 0 \} \},
\]
and $d := \min\{d_{x,t},d_{y,\tau} \}$ (by symmetry we may assume $d = d_{x,t}$).

Let us first examine the case in which $4 \rho \leq d$: under such assumption, we may assume $Q_{2\rho}(x,t) \subset Q_d(x_0,t_0) \cap Q_d(y_0,\tau_0) \cap \{ u > 0 \}$ and thus, in light of \eqref{eq:growthOfSecondDerivativesNew}, \eqref{eq:ContSecondDerivativesFB} and the definition of $\rho$, we deduce
\[
\begin{aligned}
|D^2u(x,t) - D^2u(y,\tau)| &\leq \| D^2u - D^2 q^{(x_0,t_0)} \|_{L^\infty(Q_\rho(x,t))} \\
&\quad  + \| D^2q^{(x_0,t_0)} - D^2 q^{(y_0,\tau_0)} \|_\infty \\
&\quad + \| D^2u - D^2 q^{(y_0,\tau_0)} \|_{L^\infty(Q_\rho(y,\tau))} \leq C \omega(|x-y| + \sqrt{|t-\tau|}),
\end{aligned}
\]
and the continuity of $D^2u$ follows. On the other hand, when $4 \rho \geq d$, we have $d_{y,t} \leq d + \rho$ and thus \eqref{eq:TaylorUOrder2} and \eqref{eq:ContSecondDerivativesFB} yield
\[
\begin{aligned}
|D^2u(x,t) - D^2u(y,\tau)| &\leq \| D^2u - D^2 q^{(x_0,t_0)} \|_{L^\infty(Q_d(x_0,t_0)\cap\{u>0\})} \\
&\quad  + \| D^2q^{(x_0,t_0)} - D^2 q^{(y_0,\tau_0)} \|_\infty \\
&\quad + \| D^2u - D^2 q^{(y_0,\tau_0)} \|_{L^\infty(Q_{d+\rho}(y_0,\tau_0)\cap\{u>0\})} \\
&\leq 2C\omega(d) + C \omega(d+\rho) \leq C \omega(|x-y| + \sqrt{|t-\tau|}),
\end{aligned}
\]
for some new $C > 0$ and the proof is complete.
\end{proof}
At this point, to complete the proof of Theorem \ref{thm:SmoothnessFB}, we apply the higher order boundary Harnack inequalities from \cite{K21} (as already mentioned in the introduction, we may alternatively apply \cite[Theorem 3]{KN77}): in this way, it follows that near regular points the free boundary is $C^\infty$ in space and time.
\begin{proof}[Proof of Theorem \ref{thm:SmoothnessFB}]
    The proof combines Proposition \ref{prop:C1ofFB} and the higher order boundary Harnack inequalities established in \cite{K21} as follows. 

By Proposition \ref{prop:C1ofFB}, we have 
\[
Q_{\varrho_0} \cap \{ u >0\}  = \{ (x',x_n,t) \in Q_{\varrho_0} : x_n > g(x',t) \},
\]
up to a rotation of the spatial coordinates, for some $\varrho_0 > 0$ and some $C^1_{x'} \cap C^{0,1}_t$ function $g$. Further, as already obtained in \eqref{eq:SolvPosLemma} and \eqref{eq:SolvtildePosLemma}, given any $\nu \in \Ss^n \subset \R^{n+1}$, the partial derivatives $v_\nu := \partial_\nu u$ satisfy
\[
\begin{cases}
\partial_t v_\nu - a_{ij}(x,t) \partial_{ij} v_\nu = f_\nu(x,t) \quad &\text{ in } Q_{\varrho_0} \cap\{ u > 0 \} \\
v_\nu  =  0                                                     \quad &\text{ in } Q_{\varrho_0} \cap \partial\{u>0\},	
\end{cases}
\]
for some $f_\nu \in L^\infty(Q_{\varrho_0})$ and $a_{ij}(x,t) := (DF(D^2u,x))_{ij}$. The coefficients $a_{ij}$ are continuous in $Q_{\varrho_0} \cap \overline{\{ u > 0 \}}$, by the second part of Proposition \ref{prop:C1ofFB}. On the other hand, by \eqref{eq:LipSpaceTimeFinalIneq}, we know that
\[
\partial_\nu u \geq \tfrac{1}{r} u \quad \text{ in } Q_{r\varrho_0/2}
\]
whenever $\nu = (\sigma,\pm r)$, $\sigma \in \Ss^{n-1}$ is sufficiently close to $e_n$ and $r > 0$ is small enough, up to taking $\varrho_0$ smaller: $\sigma_n \in (1-c_0,1)$ and $r \in (0,r_0/2)$, where $c_0 \in (0,1)$ is as in Lemma \ref{lem:EstimatesGlobalConvexSolutions} and $r_0 \in (0,1)$ as in Proposition \ref{prop:LipschitzRegularityOfFB}. In particular, $\partial_\nu u \geq 0$ in $Q_{\varrho_0}$ up to taking $\varrho_0$ smaller and thus, a straightforward adaptation of the proof of Lemma \ref{lem:EstimatesGlobalConvexSolutions} part (iii) shows that
\[
\partial_n u \geq c d \quad \text{ in } Q_{\varrho_0/2} \cap \{ u > 0 \}.
\]
This allows us to apply the boundary Harnack principle (see \cite[Theorem 1.2]{K21}) to the functions $v_{e_i} = \partial_i u$, $v_{e_{n+1}} := \partial_t u$ and $v_{e_n} = \partial_n u$ and deduce that
\begin{equation}\label{eq:RegQuotientsBH}
\frac{\partial_i u}{\partial_n u}  \in C^{\alpha}_p (Q_{\varrho_0/4} \cap \overline{\{ u > 0 \}}),
\end{equation}
for every $\alpha \in (0,1)$ and $i = 1,\ldots,n+1$. Consequently, the functions
\[
\bar{\nu}_i = \frac{\partial_i u}{|\nabla_{x,t} u|} = \frac{\partial_i u/\partial_nu}{\left(\sum_{j=1}^{n-1}(\partial_ju/\partial_nu)^2 + 1 + (\partial_tu/\partial_nu)^2\right)^{1/2}}, \qquad i = 1,\ldots,n+1
\]
can be $C^{\alpha}_p$-extended up to $\overline{\{ u > 0 \}}$ and so, by definition of $\bar{\nu}$ and $g$, this yields 
\[
g \in C^{1+\alpha}_p,
\]
for every $\alpha \in (0,1)$.

We can bootstrap this argument by means of the higher orders boundary Harnack inequalities \cite[Theorem 1.3]{K21}.
Assume that the free boundary is already $C^\beta_p$, for some $\beta>1$. First by \cite[Corollary 5.3]{K21} all the derivatives $\partial_\sigma u$ are $C^\beta_p$ up to the boundary, hence $D^2u$ is $C^{\beta-1}_p$. It follows that the coefficients of the equation for the derivatives are also $C^{\beta-1}_p$, which allows us to apply higher order boundary Harnack inequality \cite[Theorem 1.3]{K21} to deduce that $\frac{\partial_i u}{\partial_n u}$, $i=1,\ldots,n+1$ are $C^{\beta}_p$. Hence $g$ is $C^{\beta+1}_p$ and the claim follows.
\end{proof}

\begin{proof}[Proof of Theorem \ref{thm:1.1}]
	The result is an immediate consequence of Theorem \ref{thm:SmoothnessFB}, Lemma \ref{lem:halfParabola} and Lemma \ref{lem:blow-upSingularPoint}.
\end{proof}
We end the section with a couple of corollaries of Theorem \ref{thm:SmoothnessFB}. In the first one, we show uniqueness of blow-ups at regular points while, in the second, we prove that regular points have density $1/2$ (in the parabolic sense). 
\begin{cor}
	Let $u\in\mathcal{P}_1(K)$ and let $(x_0,t_0) \in \Reg(u)$. Then the blow-up of $u$ at $(x_0,t_0)$ is unique.
\end{cor}
\begin{proof} Up to a translation and a rotation, we may take $(x_0,t_0) = (0,0)$ and assume that $u_0(x) = a(x_n)_+^2$ is a blow-up of $u$ at $(0,0)$ along some sequence $r_k \to 0$, for some $a > 0$.

Now, assume by contradiction there is $\tilde{u}_0(x) = a(e\cdot x)_+^2$ such that $u_{\rho_k} \to \tilde{u}_0$ as $k \to \infty$ in the sense of \eqref{eq:CompactnessBlowUp}, along some new sequence $\rho_k \to 0$, for some $e \in \mathbb{S}^{n-1}$ such that $e \not= e_n$ (notice that the constant $a > 0$ is the same since, up to a rotation, $u_0$ and $\tilde{u}_0$ satisfy the same ODE, cf. Lemma \ref{lem:halfParabola}).

By Theorem \ref{thm:SmoothnessFB}, there is a smooth function $g$ such that
\[
(x,t) \in \Reg(u) \quad \text{ if and only if } \quad x_n = g(x',t)
\]
locally near $(0,0)$. Taylor expanding $g$ around the origin and re-scaling, it is not difficult to check that the normal vector $\nu_r$ to $\partial\{u_r > 0\}$ at $(0,0)$ is given by
\[
\nu_r = (-\nabla_{x'}g(0,0),1,-r\partial_tg(0,0)),
\]
for all $r \in (0,1)$. In particular, 
\[
\lim_{r \to 0} \nu_r = (-\nabla_{x'}g(0,0),1,0).
\]
On the other hand, $\nu_r = \nabla_{x,t} \, u_r/|\nabla_{x,t} \, u_r|$ and thus, a direct computation combined with \eqref{eq:LogContinuityut} and \eqref{eq:CompactnessBlowUp} shows that
\[
\lim_{k \to \infty} \nu_{r_k} = \frac{\nabla u_0}{|\nabla u_0|} = e_n \not= e = \frac{\nabla \tilde{u}_0}{|\nabla \tilde{u}_0|} = \lim_{k \to \infty} \nu_{\rho_k},
\]
in contradiction with the existence of the limit of $\nu_r$ as $r \to 0$. 
\end{proof}
\begin{cor}
	Let $u\in\mathcal{P}_1(K)$ and let $(x_0,t_0) \in \Reg(u)$. Then:
	$$\lim_{r \to 0}\frac{\left|\{u=0\}\cap Q_r(x_0,t_0)\right|}{\left|Q_r\right|} = \frac{1}{2}.$$
\end{cor}
\begin{proof} Without loss of generality we assume that $(x_0,t_0)=(0,0)$ and we let $u_0$ be the blow-up of $u$ at $(0,0)$. By Lemma \ref{lem:halfParabola} and Theorem \ref{thm:SmoothnessFB}, we may assume that $u_0(x) = a (x_n)_+^2$ for some $a > 0$. Now, since 
\[
\frac{|\{u = 0\}\cap Q_r|}{|Q_r |} = \frac{|\{u_r = 0\}\cap Q_1|}{|Q_1 |} \quad \text{ and } \quad \frac{|\{u_0 = 0\}\cap Q_1|}{|Q_1|} = \frac{1}{2},
\]
it is enough to show that
\begin{equation}\label{eq:ConChiBlowup}
\chi_{\{u_r > 0\}} \to \chi_{\{u_0 > 0\}} \quad \text{ a.e. in } Q_1, 
\end{equation}
and conclude that $|\{u_r = 0\}\cap Q_1| \to |\{u_0 = 0\}\cap Q_1|$ as $r \to 0$, by dominated convergence.

Let us show \eqref{eq:ConChiBlowup}. First, we fix $x \in \{u_0 > 0\}$ and set $u_0(x) := \delta$, for some $\delta > 0$ depending on $x$. By uniform convergence, $u_r(x) \geq \delta/2$ for all $r$ small enough and thus $\chi_{\{u_r > 0\}}(x) \to 1$ as $r \to 0$. 

Second, we show that $\chi_{\{u_r > 0\}} \to 0$ pointwise in $\{x_n < 0\} \cap Q_1$ as $r \to 0$. Assume not. Then there are $(x,t) \in \{x_n < 0\} \cap Q_1$ and $r_0 \in (0,1)$ such that $u_r(x,t) > 0$ for all $r \in (0,r_0)$. Let $\delta := \dist((x,t),\{x_n=0\})$. By construction, for every sequence $r_k \to 0$, there are $(x_k,t_k) \in \partial\{u_k > 0\}$ such that $(x_k)_n \leq -\delta$ for all $k \in \N$, where we have set $u_k := u_{r_k}$. However, by non-degeneracy \eqref{eq:NonDegeneracy}, we have
\begin{equation}\label{eq:Dens12NonDeg}
u_k(y_k,\tau_k) := \sup_{Q_{\delta/2}(x_k,t_k)} u_k \geq c \, \big(\tfrac{\delta}{2}\big)^2,
\end{equation}
for all $k \in \N$ and some $(y_k,\tau_k) \in \overline{Q_{\delta/2}(x_k,t_k)}$, where $c > 0$ is as in Lemma \ref{lem:NonDegeneracy}. Up to passing to a subsequence, we may assume $(y_k,\tau_k) \to (y_0,\tau_0)$ for some $(y_0,\tau_0) \in \{x_n \leq - \delta/2\}\cap\overline{Q}_1$. Notice that $u_0(y_0,\tau_0) = 0$ and so, by uniform convergence, $u_k(y_0,\tau_0) \to 0$ as $k \to \infty$. Consequently, by $C^{1,1}$ regularity, we have
\[
0 \leq u_k(y_k,\tau_k) \leq |u_k(y_k,\tau_k) - u_k(y_0,\tau_0)| + u_k(y_0,\tau_0) \to 0,
\]
obtaining a contradiction with \eqref{eq:Dens12NonDeg}. As a consequence, \eqref{eq:ConChiBlowup} and the proof of our statement follows.
\end{proof}
%
%
%
%
%

%
%
%
%
%
%
%
%
%
\section{Analysis of singular points}\label{sec:singularPoints}
The goal of this section is to prove Theorem \ref{thm:1.2} and Corollary \ref{cor:1.3} and establish some covering properties of the singular part of the free boundary. To derive such the properties we will exploit the fact that the second derivatives of the blow-up are positive. In this direction we prove  a second ``almost positivity'' result, in the spirit of Lemma \ref{lem:AlmostNonnegativityLemma}.

\begin{lem}\label{lem:almostPositivityImproved}
    Let $\varrho \in (0,1)$, $K > 0$ and let $\{ u_r \}_{r \in (0,1)}$ be a family of solutions to 
\begin{equation}\label{eq:EqnBlowUpSeqPosLemBis}
\begin{cases}
\partial_t v - F(D^2v) = f(rx)\chi_{\{v > 0\}} \quad  &\text{ in } Q_\varrho \\
v, \partial_t v \geq  0    \quad &\text{ in } Q_\varrho,
\end{cases}
\end{equation}
with $F$ satisfying \eqref{eq:AssOnFIntro} in $B_1$, $f$ satisfying \eqref{eq:AssOnfIntro} in $B_1$. Let $a_{ij}(x,t) := (DF(D^2u_r))_{ij}$ and let $\{ w_r \}_{r \in (0,1)}$ be a family of functions satisfying
\[
\partial_t v -  a_{ij}(x,t) \partial_{ij} v \geq r g(rx) \quad \text{ in } Q_\varrho \cap \{u_r > 0\},
\]
for some bounded function $g$ with $\|g\|_{L^\infty(B_\varrho)} \leq K$. Then there exist $\delta_0, \tilde{\vep}_0 > 0$ depending only on $n$, $c_\circ$, $\Lambda$, $K$ and $\varrho$ such that, for every $\delta \in (0,\delta_0)$, $\vep \in (0,1)$ and $\overline{C} > 0$ with $\vep/\overline{C} \leq \tilde{\vep}_0$, there exists $r \in (0,1)$ such that if
\begin{itemize}
		\item $w_r \geq 0 $ in $Q_\varrho \cap \partial\{u_r>0\}$,
		\item $w_r \geq -\varepsilon$ in $Q_\varrho$,
		\item $w_r \geq \overline{C}$ in $Q_\varrho \cap N_\delta^c(\{u_r = 0\})$,
\end{itemize}
then $w_ r \geq 0$ in $Q_{\varrho/2}$.
\end{lem}
\begin{proof}

    By scaling we may assume $\varrho = 1$. Let us set $u := u_r$, $w := w_r$, $r \in (0,1)$ and define
\[
v(x,t) := w(x,t) - \gamma \left[ u(x,t) -  \tfrac{c_\circ}{4} \left( \tfrac{1}{2n\Lambda} |x-x_0|^2 - (t-t_0) \right) \right]
\]
for some $\gamma > 0$ and $(x_0,t_0) \in Q_{1/2} \cap \{u > 0\}$. Noticing that
\[
\begin{aligned}
\partial_t v &\geq a_{ij}(x,t)\partial_{ij}w + r g(rx) - \gamma \big[ F(D^2u) + f(rx) + \tfrac{c_\circ}{4} \big] \\
\partial_{ij} v &= \partial_{ij}w  - \gamma \big[ \partial_{ij}u  - \tfrac{c_\circ}{4n\Lambda} \delta_{ij} \big],
\end{aligned}
\]
and recalling that the function $M \to F(M)$ is convex, it is not difficult to obtain  
\[
\begin{aligned}
\partial_t v - a_{ij}(x,t) \partial_{ij} v &\geq \gamma \big[ a_{ij}(x,t) - F(D^2u) \big] + r g(rx) - \gamma f(rx) - \gamma \tfrac{c_\circ}{4n\Lambda} \sum_{i=1}^n a_{ii} - \gamma \tfrac{c_\circ}{4} \\
&\geq 0 - r \|g\|_{L^\infty(B_1)} + \gamma c_\circ - \gamma \tfrac{c_\circ}{4} - \gamma \tfrac{c_\circ}{4} \\
&\geq - r K + \gamma \tfrac{c_\circ}{2} \geq 0 \quad \text{ in } Q_1 \cap \{u > 0\},
\end{aligned}
\]
provided $r$ is taken small enough depending on $c_\circ$, $\gamma$ and $K$. By the minimum principle \cite[Proposition 4.34]{IS12}, it thus follows
\begin{equation}\label{eq:MinPrincImprSemConv}
\inf_{Q_{1/2}^-(x_0,t_0) \cap \{u > 0\}} v = \inf_{\partial_p( Q_{1/2}^-(x_0,t_0) \cap \{u > 0\})}v .
\end{equation}
Now, let us take
\[
C_0 := \tfrac{64n\Lambda}{c_\circ}, \qquad \gamma := C_0\vep, \qquad \delta \leq  \sqrt{\tfrac{1}{C_0 C}}, \qquad \tfrac{\overline{C}}{\vep} \geq C_0 K,
\]
where $C > 0$ is as in \eqref{eq:OptimalGrowth}.  Then, in $Q_{1/2}^-(x_0,t_0) \cap \partial \{u > 0\}$, we have $w \geq 0$, $u = 0$ and thus
\[
v \geq \gamma\tfrac{c_\circ}{4} \left( \tfrac{1}{2n\Lambda} |x-x_0|^2 - (t-t_0) \right) \geq 0.
\]
Further, in $\partial_p(Q_{1/2}^-(x_0,t_0)) \cap N_\delta(\{u = 0\})$, we have $w \geq -\vep$ and, by optimal growth, 
\[
v \geq - \varepsilon - \gamma C \delta^2 + \gamma \tfrac{c_0}{32n\Lambda} \geq - \vep - \vep C_0 C \delta^2 + 2\vep \geq  \vep(1 - C_0 C \delta^2) \geq 0,
\]
thanks to the definitions of $\gamma$ and $\delta$. Finally, in $\partial_p(Q_{1/2}^-(x_0,t_0))\cap N_\delta( \{u = 0\})^c$, there holds
\[
v \geq \overline{C} - \gamma K  = \vep \big( \tfrac{\overline{C}}{\vep} - C_0K \big) \geq 0,
\]
by the choice of $\tfrac{\overline{C}}{\vep}$. The thesis follows by \eqref{eq:MinPrincImprSemConv} and the arbitrariness of $(x_0,t_0) \in Q_{1/2} \cap \{u > 0\}$.
\end{proof}

\subsection{Lipschitz regularity of the whole free boundary}
In what follows, we apply the ``almost positivity'' lemma to bound the gradient of the solution $\nabla u$ in terms of its time derivative $\partial_t u$, locally near any free boundary point. This fact, combined with the Implicit Function theorem, will allow us to conclude that the free boundary can be locally written as a Lipschitz graph, where time is explicit in terms of the space variables.

\begin{lem}\label{lem:towardsLipschitzness}
	 
	Let $u \in \mathcal{P}_1(K)$ with $(0,0) \in \partial \{u>0\}$. Assume that  $\|f\|_{C^{0,1}(B_1)} \leq K$ and that $\partial_t u>0$ in $\{u>0\}$. Then there exist $c > 0$ and $\varrho_0 > 0$ such that
	\begin{equation}\label{eq:BoundBelowut}
	\partial_t u \geq c \, |\nabla u|  \quad \text{ in } Q_{\varrho_0}.
	\end{equation}

\end{lem}
\begin{proof}
	  Let $\{ u_r \}_{r \in (0,1)}$ be a family of rescalings of $u$ as in \eqref{eq:RescalingBlowUp}, satisfying \eqref{eq:EqnBlowUpSeqPosLemBis} in $Q_1$, and consider the functions
	\[
	v := A \partial_t u_r \pm \partial_i u_r,
	\]
	for $i = 1,\ldots,n$ and some $A > 1$. Differentiating the equation, we easily see that $v$ satisfies
	\[
	\partial_t v - a_{ij}(x,t) \partial_{ij} v = r \partial_i f(rx) \quad \text{ in } Q_1 \cap \{ u_r > 0 \}, 
	\]
	where, as always, $a_{ij}(x,t) := (DF(D^2u_r))_{ij}$ and, in addition, $v = 0$ in $\partial\{u_r > 0\}$. Further, given a small $\delta > 0$ as in Lemma \ref{lem:almostPositivityImproved}, we exploit Theorem \ref{thm:OptimalReg} to see that  
	\[
	v \geq - C\delta := -\vep \quad \text{ in } Q_1 \cap N_\delta(\{u_r = 0\}),
	\]
	where $C > 0$ depends only on $n$, $\lambda$, $\Lambda$ and $K$. On the other hand, since $\partial_tu_r >0$ in $\{ u_r > 0 \}$, we have that $\partial_tu_r \geq \bar{c}$ in $N_\delta(\{u_r = 0\})^c \cap Q_1$ for some constant $\bar{c} > 0$ depending on $\delta$ and $r$ and hence, by Theorem \ref{thm:OptimalReg} again, we may choose $A$ large enough such that
	\[
	v \geq A\bar{c} - C \geq \tfrac{A\bar{c}}{2} := \overline{C} \quad \text{ in } Q_1 \cap N_\delta(\{u_r = 0\})^c.
	\]
	Taking eventually $\delta$ smaller and $A$ larger, we may assume both $\delta \leq \delta_0$ and $\vep/\overline{C} \leq \tilde{\vep}_0$ where $\delta_0$ and $\tilde{\vep}_0$ are as in Lemma \ref{lem:almostPositivityImproved} and, by the same lemma, we deduce the existence of $r$ such that $v \geq 0$ in $Q_{1/2}$, which is equivalent to our statement with $c := 1/A$ and $\varrho_0 := r$.
\end{proof}
Below we prove the Lipschitz regularity of the free boundary near any point.
\begin{cor}\label{cor:LipschitzFB}
	Let $u \in \mathcal{P}_1(K)$ with $(0,0) \in \partial \{u>0\}$. Assume that  $\|f\|_{C^{0,1}(B_1)} \leq K$ and that $\partial_t u>0$ in $\{u>0\}$. Then there exists a Lipschitz function $\tau: B_{1/2} \to \R$ such that
	\begin{equation}\label{eq:LipschitzFB}
	Q_{1/2}\cap\{u > 0\} = \{ (x,t) \in Q_{1/2} : t > \tau(x)\}.
	\end{equation}

\end{cor}
\begin{proof}
	 
	Let us consider the level sets $\{u = \vep\}$, for $\vep > 0$. Since $\partial_t u > 0$ in $\{u > 0\}$, we may apply the Implicit Function theorem to deduce the existence of $\varrho_0,\vep_0 >0$ and a $C^1$ function $h: B_{\varrho_0}\times(0,\vep_0) \to \R$ that locally parametrizes $\{u = \vep\}$, that is, $u(x,t) = \vep$ if and only if $t = h(x,\vep)$. Furthermore, $\partial_\vep h > 0$ and, by \eqref{eq:BoundBelowut}, we also have 
	\[
	|\partial_i h| = \frac{|\partial_i u|}{\partial_t u} \leq C,  \qquad i = 1,\dots,n,
	\]
	for some $C > 0$ independent of $\vep$, up to taking $\varrho_0$ smaller. Consequently,
	\[
	|h(x,\vep) - h(y,\vep)| \leq C |x-y|,
	\]
	for some new $C$ (still independent of $\vep$) and all $x,y \in B_{\varrho_0}$. Consequently, setting $\tau(x) := \lim_{\vep \to 0} h(x,\vep)$, we may pass to the limit as $\vep \to 0$ into the above inequality and conclude the proof of our statement.
\end{proof}

\subsection{$\varepsilon$-flatness in space of the singular set}

In order to establish further regularity in space, the idea is to apply the ``almost positivity'' lemma to the second spatial derivatives of the solution (which become supersolutions to the linearised equation, as in \eqref{eq:secondDerivsSupersol}). For this reason, in the majority of the following statements, we have to assume that $F$ in \eqref{eq:AssOnFIntro} is independent of the space variables.

We begin with an auxiliary result stating that, when we move away from the free boundary, the convergence of the blow-up sequence is actually in $C^2$.

\begin{lem}\label{lem:C2convergence}
    Let $u \in \mathcal{P}_1(K)$ with $(0,0) \in \Sigma(u)$. Let $\{u_r\}_{r \in (0,1)}$ be the family of rescalings defined in \eqref{eq:RescalingBlowUp} and assume $u_{r_k}$ is a blow-up sequence satisfying \eqref{eq:CompactnessBlowUp}. Then for every $\varepsilon,\delta > 0$, there exists $k_0 \in \N$ such that
\[ 
\|u_{r_k} - u_0 \|_{C^2(Q_{1/2}\cap N_\delta(\{u_{r_k} = 0 \})^c)}\leq \varepsilon,
\]
for all $k\geq k_0$.
\end{lem}
\begin{proof}   Similar to the proof of Proposition \ref{prop:C1ofFB}, since $u_0$ is a quadratic polynomial, $D^2u_0 := A$ is a constant matrix satisfying $-F(D^2u_0) = -F(A) = f(0)$. Now, let us set $u_k := u_{r_k}$ and $v_k := u_k - u_0$. Using the equations of $u_k$ and $u_0$, we easily see that 
\[ 
\partial_tv_k - \tilde{F}(D^2v_k) = \tilde{f}_k(x) \quad \text{in } Q_1 \cap \{u_k > 0 \},
\]
for all $k \in \N$, where $\tilde{F}(M) := F(M + A) - F(A)$ and $\tilde{f}_k(x) := f(r_kx) + F(A)$. Hence, interior estimates ( \cite[Theorem 1.1]{CK17} or \cite[Theorem 1.1]{Wang92bis}) yield
\[
\|D^2v_k\|_{L^\infty(Q_\delta(x_0,t_0))} \leq \frac{C}{\delta^2} \left( \|v_k \|_{L^\infty(Q_{2\delta}(x_0,t_0))} + \delta^2 \| \tilde{f}(r_kx) \|_{L^\infty(Q_{2\delta}(x_0,t_0))} \right),
\]
for some $C > 0$ depending only on $n$, $\lambda$ and $\Lambda$, and all $(x_0,t_0) \in Q_{1/2}$ and $\delta > 0$ such that $Q_{2\delta}(x_0,t_0) \subset Q_1\cap \{u_k > 0\}$. Consequently, given any $\vep \in (0,1)$, we may combine the above estimate with the fact that both $v_k$ and $\tilde{f}_k$ converge to zero locally uniformly in $\R^n$, to deduce the existence of $k_0 \in \N$ such that
\[
\|D^2v_k\|_{L^\infty(Q_\delta(x_0,t_0))} \leq \vep,
\]
for all $k \geq k_0$. By the arbitrariness of $(x_0,t_0) \in Q_{1/2} \cap N_\delta(\{u_k = 0\})^c$, we complete the proof of our statement.
\end{proof}
From the above result, we conclude that if the blow-up is positive in some direction, then the second derivative in space of the solution is positive along the same direction, when we look away from the boundary. On the other hand, thanks to the semiconvexity estimate \eqref{eq:LogSemiConvexity} the second derivative is not too negative close to the free boundary and hence the ``almost positivity'' lemma applies. The next lemma, in the spirit of \cite{B01}, is crucial to establish the $\vep$-flatness in space. 
\begin{lem}\label{lem:positivityOfSecondDer}
Let $u \in \mathcal{P}_1(K)$ with $F$ as in \eqref{eq:AssOnFIntro} and independent of $x$, $f$ satisfying \eqref{eq:RegAssFfSemiConv} and $(0,0) \in \Sigma(u)$. Assume also 
\begin{equation}\label{eq:SingBlowUpCanonic}
u_0(x) = \sum_{j = 1}^n \lambda_j x_j^2, \qquad \lambda_j \geq 0, \quad \lambda_n > 0,
\end{equation}
is a blow-up of $u$ at $(0,0)$. Then for every $\vartheta \in [0,\pi/2)$, there exist $\varrho_0,c>0$ depending on $\vartheta,\lambda_n$, such that 
\[
\partial_{ee}u \geq c \, |\nabla u| \quad \text{ in } Q_{\varrho_0},
\]
for every $e \in \mathbb{S}^{n-1}$ satisfying $e_n \geq \cos\vartheta$. 
\end{lem}
\begin{proof}
    Let us fix $\vartheta \in [0,\pi/2)$, $\lambda_n > 0$ and let $u_k := u_{r_k}$ be a blow-up sequence converging to $u_0$. The main idea is to apply Lemma \ref{lem:almostPositivityImproved} to the functions 
\[
v_{k,i} := \partial_{ee} u_k \pm c\partial_i u_k, \quad k \in \N, \quad i = 1,\ldots,n.
\]
Similar to the proof of Lemma \ref{lem:towardsLipschitzness}, it is not difficult to check that each $v := v_{k,i}$ satisfies 
\[
\partial_t v - a_{ij}(x,t) \partial_{ij} v \geq r_k g(r_kx) \quad \text{ in } Q_1 \cap \{u_k > 0\}, 
\]
where $g(x) := r \partial_{ee}f(x) \mp \partial_if(x)$ satisfies $\|g\|_{L^\infty(B_1)} \leq 2K$ by assumption. Furthermore, we have $v \geq 0$ in $Q_1 \cap \partial\{u_k > 0\}$ (this is an immediate consequence of the fact that $u_k \geq 0$ and $|\nabla u_k| = 0$ on $\partial\{ u_k > 0 \}$).

Now, let $e \in \Ss^{n-1}$ with $e_n \geq \cos \vartheta$. Since $\partial_{nn} u_0 = 2\lambda_n > 0$ and $\lambda_j \geq 0$ for every $j = 1,\ldots,n-1$, it is not difficult to see that $\partial_{ee} u_0 \geq \overline{C}$ for some $\overline{C} > 0$ depending on $\vartheta$ and $\lambda_n$. Consequently, in view of Lemma \ref{lem:C2convergence}, for every $\delta > 0$, there holds
\[
\partial_{ee} u_k \geq \tfrac{\overline{C}}{2} \quad \text{ in } N_\delta(\{u_k = 0\})^c,
\]
for all $k \geq k_0$ and some $k_0$ depending on $\vartheta$, $\lambda_n$ and $\delta$. If $C_0 > 0$ is such that $|\nabla u_k| \leq C_0$ in $Q_1$ ($k \geq k_0$), we may choose $c < \tfrac{\overline{C}}{4C_0}$ to obtain
\[
v \geq C_* := \tfrac{\overline{C}}{4}  \quad \text{ in } Q_1 \cap N_\delta(\{u_k =0\})^c.
\]
Finally, by Proposition \ref{prop:semicnovexity} and Lemma \ref{lem:OptimalGrowth}, we have that
\[
v \geq - C|\log \delta|^{-\varepsilon} - C \delta  \quad \text{ in } N_\delta(\{u_k = 0\}),
\]
for some new $C > 0$ and $\vep \in (0,1)$ depending only on $n$, $\lambda$, $\Lambda$ and $K$. Choosing $\delta$ small enough, the assumptions of Lemma \ref{lem:almostPositivityImproved} are fulfilled for each $k\geq k_0$, and so
\[
v \geq 0 \quad\text{ in } Q_{1/2},
\]
that is, taking $k = k_0$,
\[ 
\partial_{ee}u \geq \tfrac{c}{r_{k}} |\partial_iu| \quad \text{ in } Q_{r_{k}/2},
\]
for all $i = 1,\ldots,n$, which readily implies our claim with $\varrho_0 := r_{k_0}$.
\end{proof}
Once the second derivatives along a cone of directions are positive in a neighbourhood of a singular point, we can prove that the solution itself is positive along the same cone of directions at any other singular point in a smaller neighbourhood.
\begin{cor}\label{cor:coneProperty}
Let $u \in \mathcal{P}_1(K)$ with $F$ as in \eqref{eq:AssOnFIntro} and independent of $x$, $f$ satisfying \eqref{eq:RegAssFfSemiConv} and $(0,0) \in \Sigma(u)$. Assume also 
\[
u_0(x) = \sum_{j = 1}^n \lambda_j x_j^2, \qquad \lambda_j \geq 0, \quad \lambda_n > 0,
\]
is a blow-up of $u$ at $(0,0)$. Then for every $\vartheta \in [0,\pi/2)$, there exists $\varrho_0 > 0$ depending on $\vartheta,\lambda_n$, such that for every $(x_0,t_0)\in Q_{\varrho_0} \cap \Sigma(u)$, we have
\[ 
u > 0  \quad \text{ in } \left\{(x,t_0) \in B_{\varrho_0} : |(x-x_0)_n| > \cos \vartheta \, \|x-x_0\|  \right \}.
\]
\end{cor}
\begin{proof}
  Let us fix $\vartheta \in (0,\pi/2)$ and $\theta \in (0,\vartheta)$. By Lemma \ref{lem:positivityOfSecondDer}, there are $\varrho_0,c > 0$ depending on $\lambda_n$ and $\theta$, such that
\begin{equation}\label{eq:positivityOfSecondDerBis}
\partial_{ee}u \geq c \, |\nabla u| \quad \text{ in } Q_{\varrho_0},
\end{equation} 
for every $e \in \mathbb{S}^{n-1}$ with $e_n > \cos \theta$.

Now, by contradiction, we assume there is $(x_0,t_0) \in Q_{\varrho_0} \cap \Sigma(u)$ and $(x,t_0) \in \{ (x,t_0) : |(x-x_0)_n| > \cos \theta \, \|x-x_0\| \}$ such that $u(x,t_0) = 0$, and we proceed with a delicate geometrical construction as follows.

Let $e := (x-x_0)/\|x-x_0\|$. By \eqref{eq:positivityOfSecondDerBis}, we have $\partial_{ee}u \geq 0$ in $Q_{\varrho_0}$ and thus $u = 0$ on the segment $[x_0,x]$. Now, let us fix $x_1, x_2 \in \intrr([x_0,x])$ satisfying $|x_1 - x_2| = \ell > 0$, and let us choose a system of coordinates $y = (y',y_n)$ such that $x_1$ coincides with the new origin and $e$ is the new $n^{th}$ unit vector.

Let $\bar{e} \in \Ss^{n-1}$ be perpendicular to $e$ and $z=(1-s)x_1 + sx_2 + r \bar{e}$, where $s \in [0,1]$ and $r > 0$. Then if $r \in (0,r_0)$ and $r_0 >0$ is small enough (depending on $\theta$, $x_1$, $x_2$), we have $\partial_{e_ie_i} u \geq 0$, $i=1,2$, where  
\[
e_1 := \tfrac{x_0 - z}{\|x_0 - z\|}, \qquad e_2 := \tfrac{z-x}{\|z - x\|},
\]
and thus we may combine this observation with $u(x_0) = u(x) = 0$ to deduce
\[
\partial_{e_1}u(z) \geq 0 \quad \text{ and } \quad \partial_{e_2}u(z) \leq 0.
\]
Consequently, writing $e_i = \alpha_i(r)e + \beta_i(r)\bar{e}$ ($i = 1,2$) for some $\alpha_i,\beta_i:\R \to \R$ satisfying $\alpha_i(r) \in (1/2,1)$ in $(0,r_0)$ and $\beta_i(r) \leq C r$ in $(0,r_0)$ for some constant $C > 0$ depending on $\theta$ and $x_1$, we obtain
\[
\partial_e u(z)\geq - \tfrac{|\beta_1(r)|}{\alpha_1(r)} |\partial_{\bar{e}}u(z)| \geq -C r^2,
\]
for some new $C > 0$. Repeating the argument with $e_2$, it follows
\[
|\partial_e u(z)| \leq Cr^2.
\]
As a consequence, if $y' \in B_r'$ and $p_1 := x_1 + (y',0)$, $p_2 := x_2 + (y',0)$, we deduce
\[
\int_{p_1}^{p_2} \partial_{ee}u = \partial_eu(x_2+y')- \partial_eu(x_1+y') \leq C|y'|^2,
\]
which, combined with \eqref{eq:positivityOfSecondDerBis}, yields
\begin{equation}\label{eq:GradBoundPower2}
\int_{p_1}^{p_2} |\nabla u| \leq C|y'|^2,
\end{equation}
for some new $C > 0$. Now, we consider the cylinder
\[
\mathcal{C}_r := \{(1-s)x_1 + sx_2 + (y',0): s \in (0,1), \, y' \in B_r' \},
\]
with $r \in (0,r_0)$ and we show that
\begin{equation}\label{eq:averageR^3Bis}
\fint_{\mathcal{C}_r} u \leq Cr^3.
\end{equation}
Indeed, if $\nu(y') := y'/|y'|$, we may exploit \eqref{eq:GradBoundPower2} and that $r \leq r_0$ to estimate
\begin{align*}
\fint_{\mathcal{C}_r} u(y) dy = & \frac{c_n}{\ell r^{n-1}} \int_{\mathcal{C}_r} \int_0^{|y'|} \nabla u(s\nu(y'),y_n) \cdot \nu(y') \, ds dy \\
	\leq &\frac{c_n}{\ell r^{n-1}} \int_{B_r'} \int_0^{|y'|}\int_{p_1}^{p_2} |\nabla u|(s \nu(y'),y_n ) \, dy_n ds dy' \\
	\leq & \frac{c_n C}{\ell r^{n-1}} \int_{B_r'} \int_0^{|y'|}  s^2 ds dy' \leq \frac{c_n C}{\ell r^{n-1}} \int_{B_r'} |y'|^3 dy' \leq Cr^3,
\end{align*}
for some new $C > 0$. At this point, we take $m := \lfloor \ell/2r \rfloor$ disjoint balls inside $\mathcal{C}_r$ with centres $y_j^r \in [x_1,x_2]$, $j = 1,\ldots,m$. Since the volumes of $\cup_{j=1}^m B_r(y_j^r)$ and $\mathcal{C}_r$ are comparable by construction, the above inequality yields
\[
\fint_{\cup_{j=1}^m B_r(y_j^r)} u(y) dy \leq C r^3,
\]
for some new $C > 0$ and thus, there must be $j \in \{1,\ldots,m\}$ such that, setting $y_r := y_j^r$, we have
\[
\fint_{B_r(y_r)}u\leq Cr^3.
\]
Finally let us show that this is in contradiction with the non-degeneracy estimate \eqref{eq:NonDegeneracy}. Notice that since $\partial_t u \geq 0$, $u(\cdot,t_0)$ satisfies
$$F(D^2u(x,t_0))\geq 0. $$
Hence by the maximum principle \cite[Proposition 4.34]{IS12}, we have that 
$$\sup_{B_r(y_0)}u\leq C\fint_{B_{2r}(y_0)}u$$
for any $y_0$ and $r>0$.
But then for $y_r$ at scale $r$ we have
$$cr^2\leq \sup_{Q_r^-(y_r,t_0)}u\leq \sup_{B_r(y_r)}u(\cdot,t_0)\leq Cr^3,$$
which gives us the desired contradiction.
\end{proof}

The above result says that the free boundary can be touched at singular points with two-sided cones of arbitrarily large opening, when we look any time-slice of that singular point. Combining this property with the time monotonicity of $u$ and the fact that the free boundary is a Lipschitz graph $t = \tau(x)$, we show that the same is true also for the projection of the singular set
\[
\operatorname{pr}(\Sigma(u)) = \{x\in B_1:\text{ } (x,t)\in\Sigma(u)\text{ for some }t\in(-1,1)\}.
\]
For simplicity we denote  the two-sided  cone centred at $x_0$ of opening $\vartheta$ in direction $e_0$ with
\[
\mathcal{C}(x_0,e_0,\vartheta) := \{x\in\R^n: \text{ }\frac{\sqrt{|x-x_0|^2 - ((x-x_0)\cdot e_0)^2}}{|(x-x_0)\cdot e_0|}<\tan \vartheta\}.
\]
\begin{lem}\label{lem:projcetionConeProperty}
    Let $u \in \mathcal{P}_1(K)$ with $F$ as in \eqref{eq:AssOnFIntro} and independent of $x$, $f$ satisfying \eqref{eq:RegAssFfSemiConv} and $(0,0) \in \Sigma(u)$. Assume also  $\|f\|_{C^{0,1}(B_1)} \leq K$ and that $\partial_t u>0$ in $\{u>0\}$. Then there exists  $e_0 \in \mathbb{S}^{n-1}$ such that for every $\vartheta \in (0,\pi/2)$, there exists $\varrho_0 > 0$ such that 
\begin{equation}\label{eq:ProjectionConeEmptyinter}
B_{\varrho_0} \cap \operatorname{pr}(\Sigma(u)) \cap \mathcal{C}(x_0,e_0,\vartheta) = \emptyset
\end{equation}
for all $x_0 \in B_{\varrho_0} \cap \operatorname{pr}(\Sigma(u))$.  
\end{lem}
\begin{proof}
    Let $u_0$ be a blow-up of $u$ at $(0,0) \in \Sigma(u)$. Then, up to a rotation of the coordinates system, we may assume that $u_0$ is of the form \eqref{eq:SingBlowUpCanonic}. Consequently, by Corollary \ref{cor:coneProperty}, given any $\vartheta \in (0,\pi/2)$, we deduce the existence of $\rho_0 > 0$ such that $u > 0$ in $\mathcal{C}(x_0,e_n,\vartheta) \times \{t_0\}$, for any other singular point $(x_0,t_0)\in Q_{\rho_0}$. Furthermore, since $\partial\{u > 0\}$ can be written as the graph of a Lipschitz function $\tau = \tau(x)$ by Corollary \ref{cor:LipschitzFB}, the full free boundary in $B_{\varrho_0}\times\R$ is actually contained in $Q_{\rho_0}$, if $\varrho_0 := \tfrac{1}{L} \rho_0^2$, where $L$ is the Lipschitz constant of the function $\tau$.

Now, let $x_0 \in B_{\varrho_0} \cap \operatorname{pr}(\Sigma(u))$ and assume by contradiction there exists $y_0 \in B_{\varrho_0} \cap \operatorname{pr}(\Sigma(u)) \cap \mathcal{C}(x_0,e_n,\vartheta)$. Then, from we have just discussed above, there must be $t_0$ and $s_0$ such that $(x_0,t_0),(y_0,s_0) \in Q_{\varrho_0} \cap \Sigma(u)$ (notice that by symmetry we may assume $t_0 < s_0$) and, by Corollary \ref{cor:coneProperty}, there holds $u > 0$ in $\mathcal{C}(x_0,e_n,\vartheta) \times \{t_0\}$.  However, by time-monotonicity, the same is true in $\mathcal{C}(x_0,e_n,\vartheta)\times\{s_0\}$ which, in particular, contains $(y_0,s_0)$. This contradicts the fact that $(y_0,s_0)$ is a free boundary point.
\end{proof}
As each cone with vertex at singular points can be taken arbitrarily close to a half-space, the projection of the singular set is $\varepsilon$-flat for any $\varepsilon>0$, as stated in the next corollary.
\begin{cor}\label{cor:coveringManifold}
 Let $u \in \mathcal{P}_1(K)$ with $F$ as in \eqref{eq:AssOnFIntro} and independent of $x$, $f$ satisfying \eqref{eq:RegAssFfSemiConv} and lvet $(0,0) \in \Sigma(u)$. Assume also  $\|f\|_{C^{0,1}(B_1)} \leq K$ and that $\partial_t u>0$ in $\{u>0\}$. Then for every $\varepsilon > 0$, there exists $\varrho_0 > 0$ and a Lipschitz function $G: \R^{n-1} \to \R$ with $[G]_{C^{0,1}(\R^{n-1})} \leq \vep$ such that
\[
B_{\varrho_0} \cap \operatorname{pr}(\Sigma(u)) \subset B_{\varrho_0} \cap \operatorname{graph}(G),
\]
up to a rotation of the coordinates system.
\end{cor}
\begin{proof}
    Let us fix $\vep > 0$ and set $\vartheta := \arctan(1/\varepsilon)$. Then, up to a rotation, Lemma \ref{lem:projcetionConeProperty} yields the existence of $\varrho_0 > 0$ such that 
\[
B_{\varrho_0} \cap \operatorname{pr}(\Sigma(u)) \cap \mathcal{C}(x_0,e_n,\vartheta) = \emptyset,
\]
for any other $x_0 \in B_{\varrho_0} \cap \operatorname{pr}(\Sigma(u))$. In particular, we deduce that for every $x' \in B_{\varrho_0}'$ there is at most one $x_n$ such that $(x',x_n) \in B_{\varrho_0} \cap \operatorname{pr}(\Sigma(u))$. Now, set
\[
S = \{x': \text{ there is } x_n \text{ such that } (x',x_n) \in B_{\varrho_0} \cap \operatorname{pr}(\Sigma(u) )\},
\]
and define $G(x') = x_n$ on $S$. By the above property, $G$ is Lipschitz continuous with $[G]_{C^{0,1}(B_{\varrho_0}')} \leq \varepsilon$ and thus, by Kirszbraun's theorem, $G$ may be extended to $\R^{n-1}$ without increasing its Lipschitz seminorm. It thus follows that $B_{\varrho_0} \cap  \operatorname{pr}(\Sigma)$ is covered by the graph of (the extension of) $G$ and our statement follows.
\end{proof}
In particular, since the free boundary is Lipschitz, this implies Theorem \ref{thm:1.2}.
\begin{proof}[Proof of Theorem \ref{thm:1.2}]
	Thanks to Corollary \ref{cor:LipschitzFB} the free boundary can be written as the graph over the time coordinate of a Lipschitz function $\tau$ as in \eqref{eq:LipschitzFB}. But as for any $\varepsilon>0$ the projection of the singular set can be locally covered by a graph of a  function $G$ with $[G]_{C^{0,1}(\R^{n-1})} \leq \vep$, see Corollary \ref{cor:coveringManifold}, the full singular set is locally covered by 
	$$\Sigma\cap Q_{r}\subset\{(x',G(x'),\tau(x',G(x'))\}.$$
	The claim follows.
\end{proof}
Once the singular set can be covered by a $(n-1)$-dimensional Lipschitz manifold, it follows that the singular set cannot be too large at most times, in view of \cite[Corollary 7.8]{FigRosSerra}.
\begin{proof}[Proof of Corollary \ref{cor:1.3}]
	Because the projection of the singular set is locally contained in a Lipschitz manifold of dimension $n-1$, it has at most Hausdorff dimension $n-1$. Moreover, since the full singular set is Lipschitz, it can be touched from above by cones. Hence \cite[Corollary 7.8]{FigRosSerra} applies and yields the result.
\end{proof}

\appendix
\section{}

\begin{proof}[Proof of Lemma \ref{lem:iteratedWeakHarnack}]
	By scaling, we may assume $r=1$. Let us fix $\delta \in (0,\tfrac{1}{4})$ and $(x,t) \in B_{1-2\delta} \times (-1 + 4\delta^2, -\tfrac{3}{4})$. By Remark \ref{rem:Scaling}, it is enough to prove
	\begin{equation*}
	\bigg( \fint_{Q_\delta^+(x,t)} u^p \bigg)^{\frac{1}{p}}\leq \frac{C}{\delta^m} \big(u(0,1) + ||f||_{L^\infty(Q_1)}\big),
	\end{equation*}
	and recover \eqref{eq:ParHalfHarnackdel} by translation. The idea is to recursively apply Theorem \ref{thm:basicWeakHarnack} on cylinders $D_{\delta_k}^-(x_k,t_k)$ and $D_{\delta_k}^+(x_k,t_k)$, where the sequences $\{\delta_k\}_{k \in \N}$ and $\{(x_k,t_k)\}_{k \in \N} \subset Q_1$ will be suitably chosen.

	We set $\delta_0 := \delta$ and we choose $(x_0,t_0) \in Q_1$ such that $D_{\delta_0}^-(x_0,t_0) = Q_{\delta_0}^+(x,t)$ (that is, $x_0=x$ and $t_0 = t + 3\delta_0^2$). Further, we define
	\[
	\delta_1 := 2\delta_0, \quad x_1 := \big(1-\tfrac{2}{|x_0|}\delta_0 \big) x_0, \quad t_1 := t_0 + 3\delta_1^2 + 3\delta_0^2.
	\]
	Notice that $x_1$ belongs to the segment joining $x_0$ and the origin, with $|x_1-x_0| = \delta_1$. Now, we firstly apply \eqref{eq:WeakHarnack} with $r = 2\delta_0$ to obtain 
	\begin{equation}\label{eq:WeakHarnStep00}
	\bigg( \fint_{D_{\delta_0}^-(x_0,t_0)} u^p \bigg)^{\frac{1}{p}} \leq C \bigg(\inf_{D_{\delta_0}^+(x_0,t_0)} u + \delta_0^2||f||_{L^\infty(Q_1)}\bigg),
	\end{equation}
	for some $C > 0$ depending only on $n$, $\lambda$ and $\Lambda$. Second, we notice that there exists $\bar{x}_0 \in B_{\delta_0}(x_0)$ such that $B_{\delta_0/2}(\bar{x}_0) \subset B_{\delta_0}(x_0) \cap B_{\delta_1}(x_1)$, by definition of $\delta_1$ and $x_1$. This implies that the set $\hat{Q}_{0,1} := D_{\delta_0}^+(x_0,t_0) \cap D_{\delta_1}^-(x_1,t_1)$ satisfies
	\[
	|\hat{Q}_{0,1}| \geq |B_{\delta_0/2}| \cdot \delta_0^2 = 2^{-n}\omega_n \delta_0^{n+2}.
	\]
	Consequently,
	\[
	\inf_{D_{\delta_0}^+(x_0,t_0)} u \leq \inf_{\hat{Q}_{0,1}} u \leq \bigg( \frac{|D_{\delta_1}^-|}{|\hat{Q}_{0,1}|} \fint_{D_{\delta_1}^-(x_1,t_1)} u^p \bigg)^{\frac{1}{p}} \leq 2^{\frac{2n+2}{p}} \bigg( \fint_{D_{\delta_1}^-(x_1,t_1)} u^p \bigg)^{\frac{1}{p}}.
	\]
	Combining the above inequality with \eqref{eq:WeakHarnStep00}, it follows
	\begin{equation}\label{eq:WeakHarnack01}
	\bigg( \fint_{D_{\delta_0}^-(x_0,t_0)} u^p \bigg)^{\frac{1}{p}} \leq C \bigg\{ \bigg( \fint_{D_{\delta_1}^-(x_1,t_1)} u^p \bigg)^{\frac{1}{p}} + \delta_0^2||f||_{L^\infty(Q_1)}\bigg\},
	\end{equation}
	for some new $C > 0$ (depending only on $n$, $\lambda$ and $\Lambda$).
	
	\smallskip
	
	Then, we iterate this procedure. Set 
	\[
	\delta_{k+1} := 2\delta_k, \quad x_{k+1} := \big( 1- \tfrac{2}{|x_k|} \delta_k) x_k, \quad t_{k+1} := t_k + 3\delta_{k+1}^2 + 3\delta_k^2.
	\]
	On the lines of the argument above, it is not difficult to find 
	\[
	\bigg( \fint_{D_{\delta_k}^-(x_k,t_k)} u^p \bigg)^{\frac{1}{p}} \leq C \bigg\{ \bigg( \fint_{D_{\delta_k}^-(x_{k+1},t_{k+1})} u^p \bigg)^{\frac{1}{p}} + \delta_k^2||f||_{L^\infty(Q_1)}\bigg\},
	\]
	for every $k \in \N$, where $C$ is as in \eqref{eq:WeakHarnack01}. The iteration stops at an index $k=N-1$ for which either
	\begin{equation}\label{eq:WeakHarnIterated}
	\begin{aligned}
	\bigg( \fint_{D_{\delta_0}^-(x_0,t_0)} u^p \bigg)^{\frac{1}{p}} &\leq C^{N-1} \bigg\{ \bigg( \fint_{D_{\delta_{N-1}}^-(x_{N-1},t_{N-1})} u^p \bigg)^{\frac{1}{p}} + 2^N ||f||_{L^\infty(Q_1)} \sum_{k=0}^{N-2} (2C)^{-j}\bigg\} \\
	& \leq C^N \bigg\{ \bigg( \fint_{D_{\delta_{N-1}}^-(x_{N-1},t_{N-1})} u^p \bigg)^{\frac{1}{p}} + 2^N ||f||_{L^\infty(Q_1)} \bigg\},
	\end{aligned}
	\end{equation}
	and $D_{\delta_{N-1}}^+(x_{N-1},t_{N-1}) \not\subseteq Q_1$, or
	\[
	\bigg( \fint_{D_{\delta_0}^-(x_0,t_0)} u^p \bigg)^{\frac{1}{p}} \leq C^N \bigg\{ \inf_{D_{\delta_{N-1}}^+(x_{N-1},t_{N-1})} u + 2^N ||f||_{L^\infty(Q_1)} \bigg\},
	\]
	and $D_{2\delta_N}^-(x_N,t_N) \not\subseteq Q_1$. In both cases, it is not difficult to check that $2^N\delta_0 \sim 1$, in the sense that $a_0 \leq 2^N\delta_0 \leq b_0$, where $a_0$ and $b_0$ are two positive numerical constants (for instance, one could easily check that in the first scenario we have $t_N - t_0 = 5\delta_0^2(2^{2N} - 1)$).

	Now, let us assume that \eqref{eq:WeakHarnIterated} holds true with $D_{\delta_{N-1}}^+(x_{N-1},t_{N-1}) \not\subseteq Q_1$ (the other case can treated similarly). Then we refine the definition of $\delta_k$ and $(x_k,t_k)$, by setting
	\[
	\delta_{k+1}' := \varrho \delta_k', \quad x_{k+1}' := \big( 1- \tfrac{\varrho}{|x_k|} \delta_k') x_k', \quad t_{k+1}' := t_k' + 3(\delta_{k+1}')^2 + 3(\delta_k')^2,
	\]
	where $\varrho \in (1,2]$, $\delta_0' := \delta_0$ and $(x_0',t_0') := (x_0,t_0)$. We then repeat the iteration above up to the step $N-1$. At this point, by continuity, we may choose $\varrho := \varrho_0 \in (1,2]$ such that
	\[
	(0,1) \in \overline{D_{\delta_{N-1}'}^+(x_{N-1}',t_{N-1}')} \subseteq \overline{Q_1}
	\]
	and so, applying once more \eqref{eq:WeakHarnack} to the r.h.s. of \eqref{eq:WeakHarnIterated}, it follows
	\begin{align*}
	\bigg( \fint_{D_{\delta_0}^-(x_0,t_0)} u^p \bigg)^{\frac{1}{p}} &\leq C^{N+1} \bigg\{ \inf_{D_{\delta_{N-1}'}^+(x_{N-1}',t_{N-1}')} u + 2^N ||f||_{L^\infty(Q_1)}  \bigg\} \\
	& \leq C^{N+1} \big( u(0,1) + 2^N ||f||_{L^\infty(Q_1)} \big).
	\end{align*}
	Recalling that $2^N\delta_0 \sim 1$ (that is, $N \sim |\log_2(\delta_0)|$), we immediately see that $C^{N+1} \leq C \delta_0^{-m}$ for some $C,m > 0$ depending only on $n$, $\lambda$ and $\Lambda$, and the thesis follows.
\end{proof}

\begin{proof}[Proof of Corollary \ref{cor:weakHarnack}]
	Thanks to translation and scaling invariance of the problem we can assume that $r=1$.
	Let us fix $\delta \in (0,1/4)$ and cover $B_1\times(-1,-3/4)$ with a finite number of disjoint ``parabolic cubes'' $\tilde{Q}_\delta(x_k,t_k)$, $k \in \{1,\dots,k_\delta\}$ for some $k_\delta \in \N$, such that
	\begin{equation}\label{eq:CovMeasBd}
	|\cup_k \tilde{Q}_{\delta}(x_k,t_k)| \leq |Q_1|.
	\end{equation}
	We also consider a family of cylinders $Q_r(x_k,t_k)$, where the radius $r$ is defined as follows: 
	\[
	r := \inf\{ \varrho > 0: \tilde{Q}_\delta(x_k,t_k) \subset Q_\varrho(x_k,t_k)\}.
	\]
	By construction, $r = c_n \delta$, for some $c_n > 0$ (depending only on $n$). Consequently, by \eqref{eq:ParHalfHarnackdel}
	\[
	\bigg( \fint_{\tilde{Q}_{\delta}(x_k,t_k)} u^p \bigg)^{\frac{1}{p}} \leq \bigg( \frac{|Q_r|}{|\tilde{Q}_{\delta}|}\fint_{Q_r(x_k,t_k)} u^p \bigg)^{\frac{1}{p}} \leq \frac{C}{\delta^m} \bigg(\inf_{Q_{1/2}^+} u + ||f||_{L^\infty(Q_1)}\bigg),
	\]
	for every $k \in \{1,\dots,k_\delta\}$, where $C$, $m$ and $p$ depend only on $n$, $\lambda$ and $\Lambda$. Summing on $k$, using Jensen's inequality ($p \in (0,1)$) and recalling \eqref{eq:CovMeasBd}, we obtain
	\begin{align*}
	\frac{C}{\delta^m}\bigg( \inf_{Q_{1/2}^-}u + ||f||_{L^\infty(Q_1)} \bigg) &\geq  \frac{1}{k_\delta}\sum_{k=1}^{k_\delta} \bigg( \fint_{\tilde{Q}_{\delta}(x_k,t_k)} u^p \bigg)^{\frac{1}{p}} \geq  \bigg( \sum_{k=1}^{k_\delta} \frac{1}{k_\delta |\tilde{Q}_{\delta}|} \int_{\tilde{Q}_{\delta}(x_k,t_k)} u^p \bigg)^{\frac{1}{p}} \\
	&= |\cup_k \tilde{Q}_{\delta}(x_k,t_k)|^{-\frac{1}{p}} \bigg(\int_{\cup_k \tilde{Q}_{\delta}(x_k,t_k)} u^p \bigg)^{\frac{1}{p}} \geq \bigg(\frac{|A|}{|Q_1|}\bigg)^{\frac{1}{p}} \bigg( \fint_A u^p \bigg)^{\frac{1}{p}} \\
	&\geq \bigg(\frac{C_0\delta^{m_0}}{|Q_1|}\bigg)^{\frac{1}{p}} \bigg( \fint_A u^p \bigg)^{\frac{1}{p}},
	\end{align*}
	which yields \eqref{eq:ParHalfHarnackdelBis}.

\end{proof}

%
%
%
%
%

\end{document}